\documentclass[onefignum,onetabnum]{siamart220329}
\pdfoutput = 1

\usepackage{lipsum,bm}
\usepackage{amsfonts,wasysym}
\usepackage{graphicx}
\usepackage{epstopdf}
\usepackage{algorithmic}
\usepackage{amsopn}
\usepackage{amssymb}
\usepackage{mathtools}
\usepackage{verbatim}
\usepackage{enumitem}

\ifpdf
\DeclareGraphicsExtensions{.eps,.pdf,.png,.jpg}
\else
\DeclareGraphicsExtensions{.eps}
\fi

\newsiamremark{Remark}{Remark}
\newsiamremark{Definition}{Definition}
\newsiamthm{Corollary}{Corollary}
\newsiamthm{Assumption}{Assumption}
\newsiamthm{Theorem}{Theorem}
\newsiamthm{Proposition}{Proposition}
\newsiamthm{Lemma}{Lemma}

\newcommand{\E}{\mathbb{E}}

\headers{Enhanced convergence rates of AMIS via QMC }{Authors}

\title{Enhanced convergence rates of Adaptive Importance Sampling with recycling schemes via quasi-Monte Carlo methods\thanks{Submitted to the editors DATE.
		\funding{This work of the first and third authors were funded by the National Science Foundation of China (No. 720711119). The fourth author was funded by the Guangdong Basic and Applied Basic Research Foundation (No. 2024A1515011876 and No. 2025A1515011888). }}
  }

\author{Jianlong Chen\thanks{Department of Mathematical Sciences, Tsinghua University, Beijing 100084, People's Republic of China (\email{chen-jl22@mails.tsinghua.edu.cn}).} \and Jiarui Du\thanks{Corresponding author. School of Mathematics, South China University of Technology, Guangzhou 510641, People's Repulic of China (\email{mascut.dujiarui@mail.scut.edu.cn}).} \and
	 Xiaoqun Wang\thanks{Department of Mathematical Sciences, Tsinghua University, Beijing 100084, People's Republic of China (\email{wangxiaoqun@mail.tsinghua.edu.cn}).} \and Zhijian He\thanks{School of Mathematics, South China University of Technology, Guangzhou 510641, People's Repulic of China (\email{hezhijian@scut.edu.cn}).}  }

\ifpdf
\hypersetup{
  pdftitle={AMIS-QMC},
  pdfauthor={A and B}
}
\fi

\begin{document}

\maketitle

\begin{abstract}
This article investigates the integration of quasi-Monte Carlo (QMC) methods using Adaptive Multiple Importance Sampling (AMIS). Traditional Importance Sampling (IS) often suffers from poor performance since it heavily relies on the choice of the proposal distributions. The AMIS and the Modified version of AMIS (MAMIS) address this by iteratively refining proposal distributions and reusing all past samples through a recycling strategy. We introduce randomized QMC (RQMC) methods into the MAMIS, achieving higher convergence rates compared to Monte Carlo (MC) methods. Our main contributions include a detailed convergence analysis of the MAMIS estimator under the RQMC setting. Specifically, we establish the $L^q$ $(q \geq 2)$ error bound for the RQMC-based estimator using a smoothed projection method, which enables us to apply the H\"older's inequality in the error analysis of the RQMC-based MAMIS estimator. As a result, we prove that the root mean square error of the RQMC-based MAMIS estimator converges at a rate of $\mathcal{O}(\bar{N}_T^{-1+\epsilon})$, where $\bar{N}_T$ is the average number of samples used in each step over $T$ iterations and $\epsilon > 0$ is arbitrarily small. Numerical experiments validate the effectiveness of our method, including mixtures of Gaussians, a banana-shaped model, and Bayesian Logistic regression.
\end{abstract}



\section{Introduction}
Importance sampling (IS) is a fundamental method in the Monte Carlo (MC) methods, widely utilized to estimate expectations with respect to a target distribution that is difficult to sample directly; see \cite{Glasserman2004}. The fundamental principle of IS involves sampling from a proposal distribution and reweighting them according to the ratio of the target and proposal densities. However, the performance of IS is highly dependent on the choice of the proposal distribution. When the proposal distribution poorly approximates the target distribution, the accuracy and stability of the estimation will be compromised. This limitation has motivated the development of Adaptive Importance Sampling (AIS) as proposed by Douc et al. \cite{Douc2007} or Capp\'{e} et al. \cite{Cappe2008}, which iteratively refines the proposal distribution based on the information collected from the previous samples. By adaptively adjusting the proposal distribution to better fit the target, AIS can improve the overall efficiency of the sampling process.

Based on the foundation of AIS, the Adaptive Multiple Importance Sampling (AMIS) method proposed by Cornuet et al. \cite{Cornuet2012} further advances the methodology by incorporating a recycling strategy. Unlike the traditional AIS, which discards previous samples once the proposal distribution is updated, the AMIS retains and reweights all the past samples in each iteration. The recycling mechanism of AMIS allows for a more comprehensive exploration of the target distribution and effectively utilizes the entire history of the sampled data, leading to a more stable and accurate estimation of the expectation with respect to the target distribution.

Nevertheless, the AMIS method of \cite{Cornuet2012} lacks theoretical convergence results due to the complexity of its algorithm. Marin et al. \cite{Marin2019} proposed a Modified version of AMIS (MAMIS) by introducing a simpler recycling strategy than the AMIS, and established the consistency of MAMIS. In \cite{Bugallo2015,Cameron2014,Martino2017LAIS,Schuster2015ConsistencyImportancSamplingEstimates,Schuster2015GradientImportanceSampling}, Marin et al. found that the performance of MAMIS is similar to that of AMIS. Consequently, this paper will focus on the MAMIS algorithm.

Moreover, previous work has primarily focused on proving the convergence of AMIS under MC setting, with notable contributions from Douc et al. \cite{Douc2007} and Marin et al. \cite{Marin2019}. Douc et al. \cite{Douc2007} proved the consistency of the Adaptive Population Monte Carlo (APMC) schemes which is similar to the AMIS. Marin et al. \cite{Marin2019} proposed the MAMIS algorithm and demonstrated its consistency. However, in the experiments, they were both limited by the relatively slow convergence rate similar to that of MC.  

Quasi-Monte Carlo (QMC) is an efficient quadrature method, possessing a higher convergence order relative to MC in the computation of integrals; see \cite{Caflisch1998,bookdick2010,Niederreiter1992,owenqmc}. Numerous studies have been conducted on the application of QMC in IS, including \cite{Chopin2016,Dick2019,He2023}, all of which have achieved higher convergence rates. In addition, note that AIS can be seen as a particular case of the Population Monte Carlo (PMC) of \cite{Cappe2004PMC}. Huang et al. \cite{Huang2022PQMC} applied the QMC methods to the PMC, called PQMC. The numerical results demonstrate a significant improvement of PQMC over PMC. Based on these advancements, we believe that the QMC methods may also bring a significant improvement to the AMIS. 

Despite these advancements, it should be noted that both the integration of QMC with AIS and PQMC lack the corresponding theoretical analysis. The convergence analysis of AMIS under QMC setting remains an open question. This paper addresses this gap by introducing the randomized QMC (RQMC) methods into the MAMIS and providing a detailed analysis of the corresponding convergence order. Notably, the consistency of the APMC in \cite{Douc2007} is proven under the assumption that the number of iteration $T$ is fixed and that the number of simulations within each iteration, $N = N_1 = N_2 = \cdots = N_T$, tends to infinity. The consistency of MAMIS in \cite{Marin2019} holds when $N_1,N_2,\ldots,N_T$ is a growing but fixed sequence and $T$ tends to infinity. The convergence rate result in this paper is with respect to  $\bar{N}_T = \frac{\Omega_T}{T}$, where $\Omega_T$ is equal to $N_1+\cdots+N_T$. More precisely, we prove that the root mean square error (RMSE) rate of RQMC-based MAMIS estimator can reach $\mathcal{O}(\bar{N}_T^{-1+\epsilon})$ for an arbitrary small $\epsilon>0$. Therefore, both the previous two cases of $N_1,\ldots,N_T$ can be included in our result. 

The contributions of this paper are threefold. First, we establish the higher-order \(L^q\) \((q \geq 2)\) error bound for RQMC-based estimators, extending the existing error analysis beyond the commonly studied \(L^1\) and \(L^2\) cases. This advancement provides a more powerful tool for error analysis under the RQMC setting. Second, we introduce RQMC methods into the MAMIS framework and analyze the RMSE of MAMIS estimator based on our \(L^q\) error result, significantly enhancing the convergence rate for the RMSE of MAMIS estimator compared to traditional MC.  Finally, we establish convergence rate results under the MC setting, complementing previous work \cite{Marin2019} that focused only on consistency. Our numerical experiments validate the effectiveness of our approach in achieving faster convergence.

The paper is organized as follows. In \cref{sec:preliminaries}, we provide the necessary preliminaries of the (R)QMC methods, including our $L^q$ error result, and review the AMIS and MAMIS algorithm. Then we presents our main contributions in \cref{sec:main}, including a detailed convergence analysis of the MAMIS estimator under RQMC sampling. In \cref{sec:MC_convergence}, we establish the convergence rate results under the MC setting. In \cref{sec:experiments} , we validate the effectiveness of our method through numerical experiments, including mixtures of Gaussians, a banana-shaped model, and Bayesian Logistic regression. Finally, we conclude the paper with a summary of our findings and potential future research directions in \cref{sec:conclusions}.

\section{Preliminaries}\label{sec:preliminaries}
\subsection{The modified version of AMIS}\label{subsec:MAMIS}
Importance sampling is a variance reduction technique in the MC methods used to approximate integrals with respect to a target distribution $\Pi$ using samples drawn from a proposal distribution $Q$. More formally, when $\Pi$ has density $\pi(x)$ with respect to the Lebesgue measure $\mathrm{d}\boldsymbol{x}$,  the goal is to estimate the integral
\begin{equation}
I := \int_{\mathcal{X}} \psi(\boldsymbol{x}) \pi(\boldsymbol{x}) \mathrm{d} \boldsymbol{x}, \label{term:integral}
\end{equation}
where $\mathcal{X} \subseteq \mathbb{R}^s$ and $\psi : \mathcal{X} \to \mathbb{R}$ is a function of interest. However, drawing samples from $\Pi$ is often difficult. IS overcomes this by drawing samples $X_1, \ldots, X_N$ from a carefully chosen proposal distribution $Q$ and weighting them by the Radon-Nikodym derivative
\[
\omega_i = \frac{\pi(X_i)}{q(X_i)},
\]
where $q(x)$ is the density of $Q$ with respect to the Lebesgue measure.
The IS estimator of the integral is then
\[
\widehat{\Pi}_N^{\text{IS}}(\psi) = \frac{1}{N} \sum_{i=1}^N \omega_i \psi(X_i).
\]

The efficiency of IS depends critically on the choice of the proposal distribution $Q$. If $Q$ is not well-tuned, the performance of IS may be poor. Adaptive Importance Sampling (AIS) addresses this challenge by iteratively updating $Q$ based on the past samples; see \cite{Cappe2008, Douc2007}. Formally, we set a well-chosen parametric family of the proposal distributions, denoted by  $\{Q(\theta)\}_{\theta\in \Theta}$, where $\Theta \subseteq \mathbb{R}^{D}$ (note that, $D$ do not need to be equal to $s$) and $Q(\theta)$ has density $q(\boldsymbol{x},\theta)$ with respect to $\mathrm{d}\boldsymbol{x}$. According to a criterion such as Kullback-Leibler divergence, variance, or others, we will use the past samples to learn a better $\theta$ at stage $t$ of the AIS algorithm, say $\widehat{\theta}_t$ for $t>1$. We present the AIS algorithm in \cref{algorithm:AIS}, which has been considered by Capp\'e et al. \cite{Cappe2008} and Douc et al. \cite{Douc2007}. 

\begin{algorithm}
\caption{Adaptive Importance Sampling}
\begin{algorithmic}\label{algorithm:AIS}
\STATE \textbf{Input:} an initial parameter $\widehat{\theta}_1$, proposal distribution family $\{Q(\theta)\}$, and sample sizes $N_1, \ldots, N_T$.
\FOR{$t = 1 \to T$}
    \IF{$t>1$}
    \STATE choose the proposal distribution $Q(\widehat{\theta}_t)$ by learning $\widehat{\theta}_t$ from the past samples
    \ENDIF
    \FOR{$i = 1 \to N_t$}
        \STATE draw $X_i^t$ from $Q(\widehat{\theta}_t)$
        \STATE compute $\omega_i^t = \frac{\pi(X_i^t)}{q(X_i^t,\widehat{\theta}_t)}$
    \ENDFOR
\ENDFOR
\RETURN the samples $(X_1^T, \ldots, X_{N_T}^T)$ with weights $(\omega_1^T, \ldots, \omega_{N_T}^T)$
\end{algorithmic}
\end{algorithm}

Utilizing the samples $X_i^T$ with weights $\omega_i^T$ obtained at the final stage $T$ in \cref{algorithm:AIS}, the AIS estimator is formulated as
\begin{equation}
\widehat{\Pi}_T^{\text{AIS}}(\psi) = \frac{1}{N_T} \sum_{i=1}^{N_T} \omega_i^T \psi(X_i^T). \notag   
\end{equation}
However, during the entire algorithm, a total of $\Omega_T = N_1+ \cdots + N_T$ samples are generated, this estimator only uses the samples at the stage $T$, and all the samples before stage $T$ are wasted. To address this issue, inspired by the deterministic multiple mixture (see \cite{Owen2000} and \cite{Veach1995}), Cornuet et al. \cite{Cornuet2012} proposed the Adaptive Multiple Importance
Sampling (AMIS), which adds a recycling strategy by reweighting all past samples at each stage $t$ in the AIS framework, and proposed using the reweighted past samples to update $\theta$. Formally, this recycling strategy is reweighting all the samples $X_i^s$ $(s\leq t)$ with weights 
\begin{equation}
\widetilde{\omega}_{i}^s = \frac{\pi(X_i^s)}{\Omega_t^{-1}\sum_{l=1}^t N_l q(X_i^s, \theta_l)} \notag
\end{equation}
after each stage $t$ in the AIS framework, and finally using 
\begin{equation}
\widehat{\Pi}_{T}^{\mathrm{AMIS}}(\psi)=\frac{1}{\Omega_{T}}\sum_{t=1}^{T}\sum_{i=1}^{N_{t}}\widetilde{\omega}_{i}^t\psi(X_{i}^{t}) \notag
\end{equation}
as the estimator after stage $T$.

Marin et al. \cite{Marin2019} also emphasized this recycling strategy, but instead of reweighting the samples at each stage $t$ to update the parameters, a simpler and more analytically tractable method was suggested. The resulting Modified version of AMIS algorithm is exposed in \cref{algorithm:MAMIS-unnormalized}. We use the assumption in \cite{Marin2019}, that is, under a criterion such as Kullback-Leibler divergence, variance, or others, the parameter $\theta^{\ast}$ which we desire $\widehat{\theta}_t$ to approximate, can be written as
\begin{equation}
\theta^{\ast} = \int_{\mathcal{X}} h(x)\pi(x) \mathrm{d}\boldsymbol{x}, \label{term:ThetastarForm}
\end{equation}
where $h : \mathbb{R}^s \to \mathbb{R}^D $ is an explicitly known function. As a result, we can iterate $\widehat{\theta}_t$ by
\begin{align}
\widehat{\theta}_{t+1} = \frac{1}{N_t} \sum_{i=1}^{N_t} \frac{\pi(X_i^t)}{q(X_i^t, \widehat{\theta}_t)} h(X_i^t),
\end{align}
in which the samples $\{X_i^t\}_{i=1}^{N_t}$ are drawn from the proposal $Q(\widehat{\theta}_t)$. Then we use the estimator
\begin{equation}\label{term:MAMISEsimator}
\widehat{\Pi}_{T}^{\text{MAMIS}}(\psi)=\frac{1}{\Omega_{T}}\sum_{t=1}^{T}\sum_{i=1}^{N_{t}}\left[\frac{\pi(X_{i}^{t})}{\Omega_{T}^{-1}\sum_{l=1}^{T}N_{l}q(X_{i}^{t},\widehat{\theta_{l}})}\right]\psi(X_{i}^{t})
\end{equation}
to estimate the integral \cref{term:integral}, where $\Omega_T = N_1 + \cdots + N_T$.

\begin{algorithm}
\caption{Modified AMIS (MAMIS)}
\begin{algorithmic}\label{algorithm:MAMIS-unnormalized}
\STATE \textbf{Input:} an initial parameter $\widehat{\theta}_1$, and sample sizes $N_1, \ldots, N_T$.
\FOR{$t = 1 \to T$}
    \FOR{$i = 1 \to N_t$}
        \STATE \textbf{draw} $X_i^t$ from $Q(\widehat{\theta}_t)$
        \STATE \textbf{compute} $\omega_i^t = \frac{\pi(X_i^t)}{q(X_i^t, \widehat{\theta}_t)}$
    \ENDFOR
    \STATE compute $\widehat{\theta}_{t+1} = N_t^{-1} \sum_{i=1}^{N_t} \omega_i^t h(X_i^t)$
\ENDFOR
\STATE \textbf{set} $\Omega_T = N_1 + \cdots + N_T$
\FOR{$t = 1 \to T$}
    \FOR{$i = 1 \to N_t$}
        \STATE \textbf{update} $\omega_i^t = \frac{\pi(X_i^t)}{{\Omega_{T}^{-1} \sum_{l=1}^T N_l q(X_i^t, \widehat{\theta}_l)}}$
    \ENDFOR
\ENDFOR
\RETURN $(X_1^1, \omega_1^1), \ldots, (X_{N_1}^1, \omega_{N_1}^1), \ldots, (X_1^T, \omega_1^T), \ldots (X_{N_T}^T, \omega_{N_T}^T).$
\end{algorithmic}
\end{algorithm}

Marin et al. \cite{Marin2019} proved the consistency of the MAMIS estimator as \( T \to \infty \), but did not provide the convergence rate with respect to the sample size. In the experiments related to AMIS \cite{Bugallo2015}, it can be observed that when \( N_1 = \cdots = N_T = N \), the RMSE convergence rate of AMIS approaches \( \mathcal{O}(N^{-1/2}) \). To improve the convergence rate of MAMIS, we introduce RQMC methods and rigorously analyze the convergence rate of the MAMIS estimator under the RQMC setting. Within our theoretical framework, we can also provide the convergence rate of MAMIS under MC.

Below, we follow the analysis framework of MAMIS and combine the results of RQMC obtained by Ouyang et al. \cite{Ouyang2024} to analyze the MAMIS estimator under RQMC setting.

Since \(\hat{\theta}_t\) approximates \({\theta}^{\ast}\), intuitively, the denominator term $\Omega_{T}^{-1}\sum_{l=1}^{T}N_{l}q(X_{i}^{t},\widehat{\theta_{l}})$ in the MAMIS estimator will approximate \(\Omega_T^{-1} \sum_{k=1}^{T} N_k q(X_i^t, \theta^*) \), which is equal to \(q(X_i^t, \theta^*)\). Hence, we introduce the auxiliary estimator
\begin{equation}
\widehat{\Pi}_{T}^{\ast}(\psi)  :=  \frac{1}{\Omega_T}\sum_{t=1}^{T} \sum_{i=1}^{N_t} \frac{\pi(X_{i}^{t})}{q(X_{i}^{t},\theta^{\ast})}\psi(X_{i}^{t}), \label{def:Pi_T^{star}}
\end{equation}
then the total error of MAMIS estimator can be decomposed into two parts
\begin{equation}
|\widehat{\Pi}_T^{\mathrm{MAMIS}}(\psi) -  I| \leq |\widehat{\Pi}_T^{\mathrm{MAMIS}}(\psi) -  \widehat{\Pi}_{T}^{\ast}(\psi)| + |\widehat{\Pi}_{T}^{\ast}(\psi) -  I|.
\end{equation}
To analyze the second part, we write $\widehat{\Pi}_{T}^{\ast}(\psi)$ as 
\begin{equation}\notag
\widehat{\Pi}_{T}^{\ast}(\psi) = \sum_{t=1}^{T} \frac{N_t}{\Omega_T} \cdot \frac{1}{N_t}\sum_{i=1}^{N_t} \frac{\pi(X_{i}^{t})}{q(X_{i}^{t},\theta^{\ast})}\psi(X_{i}^{t}),
\end{equation}
which is a weighted average of $\widehat{\pi}_{t}^{\ast}(\psi)$ given by
\begin{equation}
\widehat{\pi}_{t}^{\ast}(\psi)  :=  \frac{1}{N_t} \sum_{i=1}^{N_t} \frac{\pi(X_{i}^{t})}{q(X_{i}^{t},\theta^{\ast})}\psi(X_{i}^{t}). \notag
\end{equation}
Note that $\widehat{\pi}_{t}^{\ast}(\psi)$ serves as an estimator for the integral
\begin{equation}
\int_{\mathcal{X}} \left[ \frac{\pi(\boldsymbol{x})\psi(\boldsymbol{x})}{q(\boldsymbol{x},\theta^{\ast})}q(\boldsymbol{x},\widehat{\theta}_t)\right] \mathrm{d}\boldsymbol{x}, \notag
\end{equation}
thus we denote 
\begin{equation}
I_{\psi}^{\ast}(\theta) := \int_{\mathcal{X}} \left[ \frac{\pi(\boldsymbol{x})\psi(\boldsymbol{x})}{q(\boldsymbol{x},\theta^{\ast})}q(\boldsymbol{x}, \theta) \right] \mathrm{d}\boldsymbol{x}. \label{def:OperatorI_Psi}
\end{equation}
The integral $I$ is actually equal to $I_{\psi}^{\ast}({\theta}^{\ast})$, then we have
\begin{align}
\left|\widehat{\Pi}_{T}^{\ast}(\psi) -  I\right| 
& = \left|\sum_{t=1}^{T} \frac{N_t}{\Omega_T}(\widehat{\pi}_{t}^{\ast}(\psi) - I)\right|\notag\\
& \leq \sum_{t=1}^{T} \frac{N_t}{\Omega_T}\left|\widehat{\pi}_{t}^{\ast}(\psi) - I\right|\notag\\
& \leq \sum_{t=1}^{T} \frac{N_t}{\Omega_T}\Big[\left|\widehat{\pi}_{t}^{\ast}(\psi) - I_{\psi}^{\ast}(\widehat{\theta}_t)\right| + \left|I_{\psi}^{\ast}(\widehat{\theta}_t)- I\right|\Big]\notag\\
& = \sum_{t=1}^{T} \frac{N_t}{\Omega_T}\Big[\left|\widehat{\pi}_{t}^{\ast}(\psi) - I_{\psi}^{\ast}(\widehat{\theta}_t)\right| + \left|I_{\psi}^{\ast}(\widehat{\theta}_t)- I_{\psi}^{\ast}(\theta^{\ast})\right|\Big],\notag
\end{align}
implying that the total error can be decomposed as
\begin{align}
 \left|\widehat{\Pi}_T^{\mathrm{MAMIS}}(\psi) -  I\right| 
 \leq &\left|\widehat{\Pi}_T^{\mathrm{MAMIS}}(\psi) -  \widehat{\Pi}_{T}^{\ast}(\psi)\right|  \notag\\
 + & \sum_{t=1}^{T} \frac{N_t}{\Omega_T}\Big[\left|\widehat{\pi}_{t}^{\ast}(\psi) - I_{\psi}^{\ast}(\widehat{\theta}_t)\right| + \left|I_{\psi}^{\ast}(\widehat{\theta}_t)- I_{\psi}^{\ast}(\theta^{\ast})\right|\Big]. \notag
\end{align}
Consequently, the error analysis of the MAMIS estimator attributes to three terms as
\begin{align}
&|\widehat{\Pi}_{T}^{\mathrm{MAMIS}}(\psi) - \widehat{\Pi}_{T}^{\ast}(\psi)|, \notag\\ 
&|\widehat{\pi}_{t}^{\ast}(\psi) - I_{\psi}^{\ast}(\widehat{\theta}_{t})|, \notag\\ 
&|I_{\psi}^{\ast}(\widehat{\theta}_{t}) - I_{\psi}^{\ast}(\theta^{\ast})|. \notag 
\end{align}
In our theoretical analysis, these terms are based on the parameter error
\begin{align}
\lVert\widehat{\theta}_t - \theta^{\ast}\rVert, \notag 
\end{align}
where $\lVert\cdot\rVert$ is the Euclidean norm. When applying RQMC methods, to analyze these errors, we need the growth condition that will be introduced later. 

\subsection{Quasi-Monte Carlo methods}\label{subsec:qmc methods}
QMC methods are quadrature rules used for numerical integration of functions on the unit hypercube $[0,1)^d$. The QMC methods differ from the MC methods by using deterministic low-discrepancy sequences rather than random points. The Koksma-Hlawka inequality \cite{Niederreiter1992} provides a QMC error bound as
\begin{equation}\notag
    \left|\frac{1}{n}\sum_{j = 1}^n f(\boldsymbol{u}_j) - \int _{[0,1)^d} f(\boldsymbol{u}) d\boldsymbol{u} \right| \le V_{\mathrm{HK}}(f) D^*_n\left(\{\boldsymbol{u}_1,\dots,\boldsymbol{u}_n\}\right),
\end{equation}
where $D^*_n\left(\{\boldsymbol{u}_1,\dots,\boldsymbol{u}_n\}\right)$ is the star discrepancy of the point set $\{\boldsymbol{u}_1,\dots,\boldsymbol{u}_n\}$ and $V_{\mathrm{HK}}(f)$ is the variation of $f$ in the sense of Hardy and Krause.

In order to facilitate the error estimation, RQMC is typically used, including scrambled nets, random shift lattice and so on (see \cite{owenqmc}). It involves suitably randomizing the points \({\boldsymbol{u}}_i\) while still retaining their low-discrepancy property.

\subsection{Growth condition}\label{subsec:growth condition}
To derive the convergence rate of RQMC methods, Owen \cite{Owen2006} introduced the growth condition for unbounded function defined on $[0,1)^d$ and obtained the $L^1$ error rate based on the low variation extension strategy. Based on this method, He et al. \cite{He2023} obtained the $L^2$ error rate for the same growth condition when using scrambled nets.  

Recently, Ouyang et al. \cite{Ouyang2024} refined the growth condition \cref{def:GrowthCondition} below for the unbounded functions defined on $\mathbb{R}^d$ and obtained the convergence rate based on a different method, called the projection method. Formally, to estimate the expectation $\mathbb{E}[h(X)]$, where $X$ is a standard normal random vector, we define 
\begin{equation}
\widehat{I}_n(f) := \frac{1}{n} \sum_{j=1}^{n} f \circ \Phi^{-1}(y_j) \label{term:RQMCEstimator_ToIntroducePQMC}
\end{equation}
for $f: \mathbb{R}^d \to \mathbb{R}, $ where$\{y_j\}_{j=1}^n$ is a point set in $[0,1)^d$ and $\Phi(x)$ is the cumulative distribution function (CDF) of $\mathcal{N}(0,1)$ and $\Phi^{-1}(\boldsymbol{y})$ is the inverse of $\Phi$ acting on $\boldsymbol{y}$ component-wise. We also use $\Phi(\boldsymbol{y})$ to simplify the notation, meaning that $\Phi$ acts on $\boldsymbol{y}$ component-wise. 

Denote $1:d = \{1,2,\ldots,d\}$ and define the function class (for $M,B,k>0$) as 
\begin{equation}\label{def:GrowthCondition}
 G_{e}(M, B, k) := \left\{ f \in \mathcal{S}^{d}(\mathbb{R}^{d}) : \sup_{u \subseteq 1:d} \left| \partial^{u} f(\boldsymbol{x}) \right| \leq B e^{M|\boldsymbol{x}|^{k}} \text{ for all } \boldsymbol{x} \in \mathbb{R}^{d} \right\},    
\end{equation}
 where $\partial^{u}f(x)$ denotes the mixed partial derivative of $f(x)$ with respect to $x_j$ with $j \in u$ and $f \in \mathcal{S}^{d}(\mathbb{R}^{d})$ means that  $\partial^{u}f$ is continuous for any $u \subseteq 1:d$. Using $\mathcal{U}[0,1)^d$ to denote the uniform distribution on $[0,1)^d$, Ouyang et al. \cite{Ouyang2024} established the following result.
\begin{lemma}\label{lemma:L2-convergence}
Let $\left\{y_{1},\ldots, y_{n}\right\}$ be an RQMC point set used in the estimator $\widehat{I}_{n}(f)$ given by \cref{term:RQMCEstimator_ToIntroducePQMC} such that each $y_{j}\sim \mathcal{U}[0,1)^d$ and the star discrepancy of $\left\{y_{1},\ldots, y_{n}\right\}$ satisfies
\begin{equation}
D_{n}^{*}\left(\left\{y_{1},\ldots, y_{n}\right\}\right)\leq C\frac{(\log n)^{d-1}}{n} \quad a.s., \notag   
\end{equation}
where $C$ is a constant independent of $n$.
For the function class $G_{e}(M, B, k)$ with order $0<k<2$, we have
\begin{equation}
\sup_{f \in G_e(M,B,k)} \mathbb{E}\left[\left(\widehat{I}_{n}(f)-\mathbb{E}[f(X)]\right)^{2}\right]=\mathcal{O}\left(n^{-2}(\log n)^{3 d-2}\exp\left\{2 M(8 d\log n)^{\frac{k}{2}}\right\}\right). \notag 
\end{equation}
\end{lemma}
\begin{proof}
See the proof of Corollary 4.8 in \cite{Ouyang2024}.
\end{proof}

\subsection{$L^q$ Convergence rate of RQMC methods}
\cref{lemma:L2-convergence} pertains to the $L^2$ error; in fact, we can prove that this projection method is effective for higher-order $L^q$ $(q \geq 2)$ error as well. 

The key point of obtaining \cref{lemma:L2-convergence} lies in the introduction of a smoothed projection operator $P_R$, where $R$ is the projection radius. Formally, for the one-dimensional case, we define the projection operator $\widehat{P}_R : \mathbb{R} \rightarrow \mathbb{R}$ as
\begin{equation}\notag
\widehat{P}_R(x) = 
\begin{cases} 
R, & x > R, \\
x, & x \in [-R, R], \\
-R, & x < -R.
\end{cases}
\end{equation}
and modify it by
\begin{equation}\notag
P_R(x) = 
\begin{cases} 
-R + \frac{1}{2}, & x \in (-\infty, -R], \\
\frac{1}{2}x^2 + Rx + \frac{(R-1)^2}{2}, & x \in (-R, -R + 1), \\
x, & x \in [-R + 1, R - 1], \\
-\frac{1}{2}x^2 + Rx - \frac{(R-1)^2}{2}, & x \in (R - 1, R), \\
R - \frac{1}{2}, & x \in [R, \infty).
\end{cases}
\end{equation}
for $R>1$. Then we denote $P_R(\boldsymbol{x}) = (P_R(x_1),\ldots,P_R(x_d))$ for $\boldsymbol{x} = (x_1,\ldots,x_d)$. 

Using this operator, we can decompose the error $|\widehat{I}_{n}(f)-\mathbb{E}[f(X)]|$ as 
\begin{align}
\Big|\widehat{I}_{n}(f)-\mathbb{E}[f(X)]\Big| & \leq
\Big|\widehat{I}_{n}(f) -\widehat{I}_{n}(f\circ P_R)\Big| \notag\\
&+ \Big|\widehat{I}_{n}(f\circ P_R)-\mathbb{E}[f\circ P_R(X)]\Big|+\Big|\mathbb{E}[f\circ P_R(X)]-\mathbb{E}[f(X)]\Big|. \notag
\end{align}
Thanks to the smoothing property of $P_R$, the Koksma-Hlawka inequality can be applied to the second part, and the projection properties of $P_R$ are used to address the first part and the third part. Finally, by selecting an appropriate $R$ for different sampling sizes $n$, the convergence rate of $\mathbb{E}\left[\left(\widehat{I}_{n}(f)-\mathbb{E}[f(X)]\right)^{2}\right]$ with respect to $n$ can be achieved.

We note that this idea can be applied to obtain the higher-order $L^q$ error rate. We state the generalized theorem as follows, with the proof provided in the appendix.
\begin{theorem}\label{thm:Lq-convergence}
Given $q \geq 2$, for the function class $G_{e}(M, B, k)$ with order $0<k<2$ and the RQMC point set $\left\{y_{1},\ldots, y_{n}\right\}$ satisfying the conditions in \cref{lemma:L2-convergence}, we have
\begin{equation}
\sup_{f \in G_e(M,B,k)} \E\left[\left\lvert \widehat{I}_{n}(f)-\E[f(X)]\right\lvert^{q}\right]=\mathcal{O}\left(n^{-q}(\log n)^{q\left(\frac{3 d}{2}-1\right)} \exp \left\{q M(4 q d \log n)^{\frac{k}{2}}\right\}\right),\notag 
\end{equation}
where the constant in the big-$\mathcal{O}$ bounds depends on $M,B,k,d,q$.
\end{theorem}
\begin{proof}
See \cref{appendix:Proof of Lq-convergence}.
\end{proof}
\begin{Remark}
Note that for any $\epsilon>0$ and $0<k<2$, let $t=\log n$, we have
\begin{align}
&\lim_{n \to \infty} n^{-\epsilon}(\log n)^{q(\frac{3 d}{2}-1)}  \exp \left\{ qM \left( {4qd \log n} \right)^\frac{k}{2} \right\} \notag\\
= & \lim_{t \to \infty} t^{q(\frac{3 d}{2}-1)}\exp \left\{ qM (4qd)^{\frac{k}{2}} t^{\frac{k}{2}} - \epsilon t \right\} =0. \notag 
\end{align}
Hence the convergence rate in \cref{thm:Lq-convergence} is $\mathcal{O}(n^{-q+\epsilon})$.
\end{Remark}

 This result is crucial for our subsequent application of the H\"older's inequality to analyze the convergence of the MAMIS estimator; see \cref{term:byHolder}.

\section{Convergence rate of the RQMC-based MAMIS estimator}\label{sec:main}
In this section, we analyze the four errors considered in \cref{subsec:MAMIS} and derive the RMSE rate of the MAMIS estimator \cref{term:MAMISEsimator} under RQMC setting.

Denote $F_{\theta}:\mathbb{R}^d \to \mathbb{R}^s$ as a generator corresponding to the distribution $Q(\theta)$, that is, the random variabe $F_{\theta}(U) $ will follow the distribution $ Q(\theta)$ if $U$ follows the uniform distribution on the $d$-dimensional unit cube $[0,1)^d$. Since the growth condition in \cref{subsec:growth condition} is proposed for the integral problem with respect to the standard normal distribution, we construct $T_{\theta} := F_{\theta} \circ \Phi$, a transport map from $\mathcal{N}(\boldsymbol{0},I_{d})$ to $Q(\theta)$. Let $X_i^t = F_{\widehat{\theta}_t}(u_i^t)$ be the RQMC-based samples using in the MAMIS estimator \cref{term:MAMISEsimator}, where $\{u_i^t\}_{i=1}^{N_t}$ is an RQMC point set satisfying the conditions in \cref{lemma:L2-convergence}. Then the RMSE between the RQMC-based MAMIS estimator and the integral \cref{term:integral} is 
\begin{equation}\notag 
\Big(\E|\widehat{\Pi}_T^{\mathrm{MAMIS}}(\psi) - I|^2\Big)^{\frac{1}{2}},
\end{equation}
where the estimator $\widehat{\Pi}_{T}^{\text{MAMIS}}(\psi)$ is defined in \cref{term:MAMISEsimator}.

In order to establish our results, the following assumptions are necessary.
\begin{Assumption}\label{assumption:all}
Denote  
\begin{align}
m(\boldsymbol{x}) := \inf_{\theta \in \Theta} q(\boldsymbol{x},\theta),  \notag 
\end{align}
we assume that $q(\cdot,\theta^{\ast})$ is bounded and the following conditions are satisfied,
\begin{align}
&m(\boldsymbol{x})>0, \label{assumption:min-positive}\\
&\left|q\left(\boldsymbol{x}, \theta_1\right)-q\left( \boldsymbol{x}, \theta_2\right) \right| \leq L\left\|\theta_1-\theta_2\right\|, \forall \boldsymbol{x} \in \mathcal{X}, \label{assumption:lipschitz} \\
&\exists \eta>0, \text{ such that} \int_{\mathcal{X}}\left|\frac{\pi\psi}{mq(\cdot,\theta^{\ast})} (\boldsymbol{x})\right|^{2+\eta}\mathrm{d} \boldsymbol{x}<\infty. \label{assumption:moment}   
\end{align}    
\end{Assumption}

Let $\bar{N}_T = \frac{\Omega_T}{T}$ be the average number of samples used in each step over $T$ iterations and denote $h_j$ as the $j$-th component of $h$. We claim our result here and use the remaining four subsections to prove it.

\begin{theorem}\label{thm:MAMIS-RQMC}
Under \cref{assumption:all}, if for some $0<k<2$,  $\frac{\pi \psi}{q(\cdot,\theta^{\ast})}\circ T_{\theta} \in G_e(M,B,k)$ and $\frac{\pi h_{j}}{q(\cdot,\theta^{\ast})}\circ T_{\theta} \in G_e(M,B,k)$ $(\forall \theta \in \Theta, j = 1,\ldots,D)$, then for any $\epsilon>0$, the RQMC-based MAMIS estimator has an RMSE as
\begin{align}
 \Big(\mathbb{E}\Big|\widehat{\Pi}_T^{\mathrm{MAMIS}}(\psi) - I\Big|^2\Big)^{\frac{1}{2}} 
= \mathcal{O}(\bar{N}_T^{-1+\epsilon}), \notag
\end{align}
where the constant in the big-$\mathcal{O}$ bound depends on $M, B, k, d, D,L,\eta, \epsilon$.
\end{theorem}

\begin{corollary}
When $N_1 = N_2 = \cdots = N_T = N$, under the same condition as in \cref{thm:MAMIS-RQMC}, the RQMC-based MAMIS estimator has an RMSE as
\begin{equation}
\left(\mathbb{E}|\widehat{\Pi}_T^{\mathrm{MAMIS}}(\psi) - I|^2\right)^{\frac{1}{2}} = \mathcal{O}(N^{-1+\epsilon}),  \notag 
\end{equation}
where the constant in the big-$\mathcal{O}$ bound depends on $M, B, k, d, D,L, \eta, \epsilon$.
\end{corollary}
\begin{Remark}
Unfortunately, this corollary does not include the relationship between error and the number of iteration $T$. This is because $T$ actually appears in the constant term. In order to estimate the sum of errors at each stage, we use the same function class $G_e(M,B,k)$ to encompass all functions ${\pi \psi}/{q(\cdot,\theta^{\ast})}\circ T_{\theta}, {\pi h_{j}}/{q(\cdot,\theta^{\ast})}\circ T_{\theta}$, which eliminated the influence of $T$ in the constant term. Moreover, \cref{thm:MAMIS-RQMC} is not entirely devoid of the impact of $T$. When we use a strictly increasing sample size $\{N_t\}$, such as $N_t = Nt$, we have $\bar{N}_T=\frac{N(T+1)}{2}$. Applying \cref{thm:MAMIS-RQMC} yields an error order as $\mathcal{O}((NT)^{-1+\epsilon})$.
\end{Remark}
\begin{Remark}\label{remark:PsiCanBeVector}
The convergence order in \cref{thm:MAMIS-RQMC} still holds for the case where the integrand \( \psi \) is a finite-dimensional vector-valued function. This can be shown by applying \cref{thm:MAMIS-RQMC} to each component function of \( \psi \) and noting that the square of the vector norm is equal to the sum of the squares of its components. In this case, the constant term in \cref{thm:MAMIS-RQMC} will also be related to the dimension of the vector-valued function.
\end{Remark}
\subsection{Convergence of the parameters $\widehat{\theta}_t$}
For the parameter error $\lVert\widehat{\theta}_t - \theta^{\ast}\rVert$, let $h_{\theta}(\boldsymbol{x}) = \frac{\pi(\boldsymbol{x})h(\boldsymbol{x})}{q(\boldsymbol{x},\theta)}$, then
\begin{equation}
\Big\lVert \widehat{\theta}_{t+1} - \theta^{\ast} \Big\rVert = \Big\lVert\frac{1}{N_t} \sum_{i=1}^{N_t} h_{\widehat{\theta}_t} \circ F_{\widehat{\theta}_t}(u_i^t) - \int_{[0,1)^d} h_{\widehat{\theta}_t} \circ F_{\widehat{\theta}_t}(\boldsymbol{u})\mathrm{d}\boldsymbol{u} \Big\rVert. \notag    
\end{equation}
To analyze the convergence of it, we need only consider the convergence of its respective components.
We write $\theta^{\ast} = (\theta^{\ast}_1,\ldots,\theta^{\ast}_D)^{\top}, h(\boldsymbol{x}) = (h_1(\boldsymbol{x}),\ldots,h_D(\boldsymbol{x}))^{\top},$ and $\widehat{\theta}_{t+1} = (\widehat{\theta}_{t+1,1},\ldots,\widehat{\theta}_{t+1,D})^{\top}$, $h_{\theta}(\boldsymbol{x}) = (h_{\theta}^{(1)}(\boldsymbol{x}),\ldots,h_{\theta}^{(D)}(\boldsymbol{x}))^{\top}, $ then for any $j=1,\ldots,D,$
\begin{equation}
\Big|\widehat{\theta}_{t+1,j} - \theta^{\ast}_j\Big|
 = \Big|\frac{1}{N_t} \sum_{i=1}^{N_t} h_{\widehat{\theta}_t}^{(j)} \circ F_{\widehat{\theta}_t}(u_i^t) - \int_{[0,1)^d} h_{\widehat{\theta}_t}^{(j)} \circ F_{\widehat{\theta}_t}(\boldsymbol{u})\mathrm{d}\boldsymbol{u}\Big|. \label{term:theta^(j)}
\end{equation}

 From \cref{lemma:L2-convergence}, we obtain the following convergence result for the parameter error $\lVert\widehat{\theta}_t - \theta^{\ast}\rVert$.
\begin{theorem}\label{thm:ThetaError}
Given $q\geq 2$, if for some $0<k<2, \frac{\pi h_{j}}{q(\cdot,\theta^{\ast})}\circ T_{\theta} \in G_e(M,B,k)$ $(\forall \theta \in \Theta, j = 1,\ldots,D)$, where $T_{\theta} = F_{\theta} \circ \Phi$ , then we have
\begin{equation}
\left(\mathbb{E}{\left|\widehat{\theta}_{t+1,j} - \theta^{\ast}_j\right|}^q\right)^{\frac{1}{q}}
= \mathcal{O}\left(N_t^{-1}(\log N_t)^{\frac{3d}{2}-1}\exp{\{M(4qd\log N_t)^{\frac{k}{2}}\}}\right), \notag
\end{equation}
where the constant in the big-$\mathcal{O}$ bound depends on $M,B,k,d,q$.
Thus we can obtain
\begin{equation}
\left(\mathbb{E}{\lVert\widehat{\theta}_{t+1} - \theta^{\ast}\rVert}^q\right)^{\frac{1}{q}}
= \mathcal{O}\left(N_t^{-1}(\log N_t)^{\frac{3d}{2}-1}\exp{\{M(4qd\log N_t)^{\frac{k}{2}}\}}\right), \label{term:ThetaError}
\end{equation}
where the constant in the big-$\mathcal{O}$ bound depends on $M,B,k,d,D,q$.
\end{theorem}
\begin{proof}
From \cref{term:theta^(j)}, the conditional expectation
\begin{align} 
& \mathbb{E}\left[{\left|\widehat{\theta}_{t+1,j} - \theta^{\ast}_j\right|}^q \bigg| \widehat{\theta_t}\right] \notag\\
= & \E\left[\Big|\frac{1}{N_t} \sum_{i=1}^{N_t} h_{\widehat{\theta}_t}^{(j)} \circ F_{\widehat{\theta}_t}(u_i^t) - \int_{[0,1)^d} h_{\widehat{\theta}_t}^{(j)} \circ F_{\widehat{\theta}_t}(\boldsymbol{u})\mathrm{d}\boldsymbol{u}\Big|^q \bigg| \widehat{\theta}_t \right] \notag \\
= & \E\left[\Big|\frac{1}{N_t} \sum_{i=1}^{N_t} h_{\widehat{\theta}_t}^{(j)} \circ T_{\widehat{\theta}_t} \circ \Phi^{-1}(u_i^t) - \int_{[0,1)^d} h_{\widehat{\theta}_t}^{(j)} \circ T_{\widehat{\theta}_t} \circ \Phi^{-1}(\boldsymbol{u})\mathrm{d}\boldsymbol{u}\Big|^q \bigg| \widehat{\theta}_t \right] \notag \\
= & \E\left[\Big|\widehat{I}_{N_t}(h_{\widehat{\theta}_t}^{(j)} \circ T_{\widehat{\theta}_t})- \E[h_{\widehat{\theta}_t}^{(j)} \circ T_{\widehat{\theta}_t}(X)]\Big|^q \bigg| \widehat{\theta}_t \right], \notag  
\end{align}
where $X$ is a standard normal random vector. When $\widehat{\theta}_{t} = \theta$ is given, using $h^{(j)}_{\theta} \circ T_{\theta} = \frac{\pi h_j}{q(\cdot,\theta^{\ast})}\circ T_{\theta} \in G_e(M,B,k)$ and \cref{thm:Lq-convergence}, we have 
\begin{align}
& \E\left[|\widehat{I}_{N_t}(h_{\widehat{\theta}_t}^{(j)} \circ T_{\widehat{\theta}_t})- \E[h_{\widehat{\theta}_t}^{(j)} \circ T_{\widehat{\theta}_t}(X)]|^q \bigg| \widehat{\theta}_t = \theta \right] \notag\\
= & \E\left[|\widehat{I}_{N_t}(h^{(j)}_{\theta} \circ T_{\theta})- \E[h^{(j)}_{\theta} \circ T_{\theta}(X)]|^q \bigg| \widehat{\theta}_t = \theta \right]\notag\\
= & \E\left[|\widehat{I}_{N_t}(h^{(j)}_{\theta} \circ T_{\theta})- \E[h^{(j)}_{\theta} \circ T_{\theta}(X)]|^q \right] \label{term:HandleConditionalExpectation}\\
\lesssim & N_t^{-q}(\log N_t)^{q(\frac{3d}{2}-1)}\exp{\{qM(4qd\log N_t)^{\frac{k}{2}}\}} \notag
\end{align}
for any $\theta \in \Theta$, where the symbol $\lesssim$ is used for hiding a constant independently of $N_t$ and \cref{term:HandleConditionalExpectation} employs the independence between the RQMC point set $\{u_i^t\}_{i=1}^{N_t}$ and $\widehat{\theta}_t$. Thus we obtain 
\begin{equation}
\mathbb{E}\left[{\left|\widehat{\theta}_{t+1,j} - \theta^{\ast}_j\right|}^q \Big\vert \widehat{\theta_t}\right] \lesssim N_t^{-q}(\log N_t)^{q(\frac{3d}{2}-1)}\exp{\{qM(4qd\log N_t)^{\frac{k}{2}}\}}. \notag\\
\end{equation}

Taking the expectation on both sides of this inequality, we derive the first inequality in the theorem.

Then we can obtain the sencond inequality in the theorem as
\begin{align}
\left(\mathbb{E}{\lVert\widehat{\theta}_{t+1} - \theta^{\ast}\rVert}^q\right)^{\frac{1}{q}}
& \leq \sum_{j=1}^D \left(\mathbb{E}{\lVert\widehat{\theta}_{t+1,j} - \theta^{\ast}_j\rVert}^q\right)^{\frac{1}{q}} \notag\\
& \lesssim N_t^{-1}(\log N_t)^{\frac{3d}{2}-1}\exp{\{M(4qd\log N_t)^{\frac{k}{2}}\}}.  \notag
\end{align}
\end{proof}

\subsection{On the property of operator $I_{\psi}^{\ast}$}
For the error term $|I_{\psi}^{\ast}(\widehat{\theta}_{t}) - I_{\psi}^{\ast}(\theta^{\ast})|$, to simplify the notation, we denote $g(\boldsymbol{x}) := \frac{\pi(\boldsymbol{x})\psi(\boldsymbol{x})}{q(\boldsymbol{x}, \theta^{\ast})}$, then we can write $I_{\psi}^{\ast}(\theta)$ in \cref{def:OperatorI_Psi} as
\begin{equation}
I_{\psi}^{\ast}(\theta) = \int_{\mathcal{X}} \left[ \frac{\pi(\boldsymbol{x})\psi(\boldsymbol{x})}{q(\boldsymbol{x},\theta^{\ast})}q(\boldsymbol{x}, \theta) \right] \mathrm{d}\boldsymbol{x}
 = \int_{\mathcal{X}}  g(\boldsymbol{x})q(\boldsymbol{x}, \theta)  \mathrm{d}\boldsymbol{x}. \notag
\end{equation}

Let \( \|\cdot\|_{\infty} \) denote the essential supremum, under \cref{assumption:all}, using the H\"older's inequality, we know 
\begin{align}
\int_{\mathcal{X}} {\left|g(\boldsymbol{x})\right|} \mathrm{d}\boldsymbol{x}
\leq & \int_{\mathcal{X}} {\left|\frac{g(\boldsymbol{x})}{m(\boldsymbol{x})}\right|} {q(\boldsymbol{x}, \theta^{\ast})} \mathrm{d}\boldsymbol{x} \notag\\
\leq & \left(\int_{\mathcal{X}} {\left|\frac{g(\boldsymbol{x})}{m(\boldsymbol{x})}\right|}^{2+\eta} \mathrm{d}\boldsymbol{x}\right)^{\frac{1}{2+\eta}} \left(\int_{\mathcal{X}} q^{\frac{2+\eta}{1+\eta}}(\boldsymbol{x}, \theta^{\ast}) \mathrm{d}\boldsymbol{x}\right)^{\frac{1+\eta}{2+\eta}} \notag\\
\leq & \left(\int_{\mathcal{X}} {\left|\frac{g(\boldsymbol{x})}{m(\boldsymbol{x})}\right|}^{2+\eta} \mathrm{d}\boldsymbol{x}\right)^{\frac{1}{2+\eta}} \left( {\lVert q(\cdot,\theta^{\ast}) \rVert}_{\infty}^{\frac{1}{1+\eta}} \int_{\mathcal{X}} q(\boldsymbol{x}, \theta^{\ast}) \mathrm{d}\boldsymbol{x}\right)^{\frac{1+\eta}{2+\eta}} \notag\\
= & \left(\int_{\mathcal{X}} {\left|\frac{g(\boldsymbol{x})}{m(\boldsymbol{x})}\right|}^{2+\eta} \mathrm{d}\boldsymbol{x}\right)^{\frac{1}{2+\eta}}  {\lVert q(\cdot,\theta^{\ast}) \rVert}_{\infty}^{\frac{1}{2+\eta}}  \notag\\
< & \infty, \notag
\end{align}
where $\eta$ appears in assumption \cref{assumption:moment} and we use this assumption in the last inequality.
Thus we have
\begin{align}
\left|I_{\psi}^{\ast}(\widehat{\theta}_t) - I_{\psi}^{\ast}(\theta^{\ast})\right|
& = \left| \int_\mathcal{X} g(\boldsymbol{x})q(\boldsymbol{x},\widehat{\theta}_t)  \mathrm{d}\boldsymbol{x} - \int g(\boldsymbol{x})q(\boldsymbol{x},\theta^{\ast}) \mathrm{d}\boldsymbol{x} \right| \notag\\
& \leq  \int_\mathcal{X} \left|g(\boldsymbol{x})\right||q(\boldsymbol{x},\widehat{\theta}_t)  \mathrm{d}\boldsymbol{x} - q(\boldsymbol{x},\theta^{\ast})| \mathrm{d}\boldsymbol{x} \notag\\
& \leq L \int_\mathcal{X} {\left|g(\boldsymbol{x})\right|} \lVert\widehat{\theta}_t - \theta^{\ast}\rVert \mathrm{d}\boldsymbol{x} \notag\\
& =L \int_\mathcal{X} {\left|g(\boldsymbol{x})\right|} \mathrm{d}\boldsymbol{x} \lVert\widehat{\theta}_t - \theta^{\ast} \rVert, \label{term:Operator property}
\end{align}
which is related to the convergence of $\widehat{\theta}_t$. According to \cref{thm:ThetaError}, we obtain the following result for the error term $|I_{\psi}^{\ast}(\widehat{\theta}_{t}) - I_{\psi}^{\ast}(\theta^{\ast})|$.
\begin{theorem}\label{thm:AuxiliaryIntegralError}
Under \cref{assumption:all}, we have
\begin{align}
\Big(\mathbb{E}\big|I_{\psi}^{\ast}(\widehat{\theta}_t) - I_{\psi}^{\ast}(\theta^{\ast})\big|^2\Big)^{\frac{1}{2}} = \mathcal{O}\left(N_t^{-1}(\log N_t)^{\frac{3d}{2}-1}\exp{\{M(8d\log N_t)^{\frac{k}{2}}\}}\right), \notag
\end{align}
where the constant in the big-$\mathcal{O}$ constant depends on $M,B,k,d,D,L$.
\end{theorem}
\begin{proof}
By \cref{term:Operator property}, we can bound $\Big(\mathbb{E}\big|I_{\psi}^{\ast}(\widehat{\theta}_t) - I_{\psi}^{\ast}(\theta^{\ast})\big|^2\Big)^{\frac{1}{2}}$ as
\begin{align}
\Big(\mathbb{E}\big|I_{\psi}^{\ast}(\widehat{\theta}_t) - I_{\psi}^{\ast}(\theta^{\ast})\big|^2\Big)^{\frac{1}{2}}
&\leq L\int_\mathcal{X} {\left|g(\boldsymbol{x})\right|} \mathrm{d}\boldsymbol{x} \left(\mathbb{E}\left\lVert\widehat{\theta}_t - \theta^{\ast} \right\rVert^2\right)^{\frac{1}{2}}\notag\\
& \lesssim  N_t^{-1}(\log N_t)^{\frac{3d}{2}-1}\exp{\{M(8d\log N_t)^{\frac{k}{2}}\}}, \notag
\end{align}
where we use $\int_\mathcal{X} {\left|g(\boldsymbol{x})\right|} \mathrm{d}\boldsymbol{x} < \infty$ and \cref{thm:ThetaError} in the last inequality.
\end{proof}

\subsection{Convergence of the auxiliary estimator}
For the error term $|\widehat{\pi}_{t}^{\ast}(\psi) - I_{\psi}^{\ast}(\widehat{\theta}_{t})|$, similar to the proof of \cref{thm:ThetaError}, we can obtain the following result.
\begin{theorem}\label{thm:pi-convergence}
If for some $0<k<2, \frac{\pi \psi}{q(\cdot,\theta^{\ast})}\circ T_{\theta} \in G_e(M,B,k)$ $(\forall \theta \in \Theta)$, where $T_{\theta} = F_{\theta} \circ \Phi$ , we have 
\begin{equation}
\left(\mathbb{E}|\widehat{\pi}_{t}^{\ast}(\psi) - I_{\psi}^{\ast}(\widehat{\theta}_{t})|^2\right)^{\frac{1}{2}} = \mathcal{O}\left(N_t^{-1}(\log N_t)^{\frac{3d}{2}-1}\exp{\{M(8d\log N_t)^{\frac{k}{2}}\}}\right), \notag
\end{equation}
where the constant in the big-$\mathcal{O}$ constant depends on $M,B,k,d$.
\end{theorem}

\begin{corollary}\label{cor:Auxiliary-discrepancy}
Under the same condition as in \cref{thm:pi-convergence}, for any $\epsilon>0$, we have
\begin{align}
\left(\mathbb{E}|\widehat{\Pi}_{T}^{\ast}(\psi) - I|^2\right)^{\frac{1}{2}} = \mathcal{O}\left(\bar{N}_T^{-1+\epsilon}\right), \notag   
\end{align}
where the constant in the big-$\mathcal{O}$ constant depends on $M,B,k,d,\epsilon$.
\end{corollary}
\begin{proof}
By the triangle inequality of norm $(\mathbb{E}|\cdot|)^{\frac{1}{2}}$, for any $0<\epsilon<1$ we have
\begin{align}
\left(\mathbb{E}|\widehat{\Pi}_{T}^{\ast}(\psi) - I|^2\right)^{\frac{1}{2}}
= & \left(\mathbb{E}\left|\sum_{t=1}^T \frac{N_t}{\Omega_T} \widehat{\pi}_{t}^{\ast}(\psi) - I_{\psi}^{\ast}(\theta^{\ast})\right|^2\right)^{\frac{1}{2}} \notag\\
\leq & \sum_{t=1}^T \frac{N_t}{\Omega_T}\left(\mathbb{E}\left| \widehat{\pi}_{t}^{\ast}(\psi) - I_{\psi}^{\ast}(\theta^{\ast})\right|^2\right)^{\frac{1}{2}} \notag\\
\leq & \sum_{t=1}^T \frac{N_t}{\Omega_T}\left[\left(\mathbb{E}\left| \widehat{\pi}_{t}^{\ast}(\psi) - I_{\psi}^{\ast}(\widehat{\theta}_t)\right|^2\right)^{\frac{1}{2}} + \left(\mathbb{E}\left| I_{\psi}^{\ast}(\widehat{\theta}_t) - I_{\psi}^{\ast}(\theta^{\ast}) \right|^2\right)^{\frac{1}{2}}  \right]\notag\\
\lesssim & \frac{1}{{\Omega_T}}\sum_{t=1}^T (\log N_t)^{\frac{3d}{2}-1}\exp{\{M(8d\log N_t)^{\frac{k}{2}}\}} \label{term:usingpi-convergence}\\
\lesssim & \frac{1}{{\Omega_T}}\sum_{t=1}^T N_t^{\epsilon} \lesssim  (\frac{\Omega_T}{T})^{-1+\epsilon} = \bar{N}_T^{-1+\epsilon}, \notag
\end{align}
where \cref{term:usingpi-convergence} is derived from \cref{thm:pi-convergence} combined with \cref{thm:AuxiliaryIntegralError} and the last inequality is obtained by applying the Jensen's inequality to the concave function $f(x) = x^{\epsilon}$.
\end{proof}

\subsection{Discrepancy between the MAMIS estimator and the auxiliary estimator}
For the error term $|\widehat{\Pi}_{T}^{\mathrm{MAMIS}}(\psi) - \widehat{\Pi}_{T}^{\ast}(\psi)|$, let 
\begin{align}
\xi_T(\boldsymbol{x}) &:= \frac{1}{\Omega_T} \sum_{l=1}^T N_l q\left(\boldsymbol{x}, \hat{\theta}_l\right). \notag 
\end{align}
 By substituting \cref{term:MAMISEsimator} and \cref{def:Pi_T^{star}}, it follows that
\begin{align}
& \left(\mathbb{E}\left|\widehat{\Pi}_T^{\text{MAMIS}}(\psi)-\widehat{\Pi}_T^*(\psi)\right|^2\right)^{\frac{1}{2}} \notag\\
= & \left(\mathbb{E}\left|\frac{1}{\Omega_T} \sum_{t=1}^T \sum_{i=1}^{N_t} \frac{\pi\left(X_i^t\right) \psi\left(X_i^t\right)}{\xi_T\left(X_i^t\right)}-\frac{1}{\Omega_T} \sum_{t=1}^T \sum_{i=1}^{N_t} \frac{\pi\left(X_i^t\right) \psi\left(X_i^t\right)}{q\left(X_i^t, \theta^*\right)}\right|^2\right)^{\frac{1}{2}}  \notag\\
\leq & \frac{1}{\Omega_T} \sum_{t=1}^T \sum_{i=1}^{N_t}\left({\mathbb{E}\left|\frac{\pi\left(X_i^t\right) \psi\left(X_i^t\right)}{\xi_T\left(X_i^t\right) q\left(X_i^t, \theta^{\ast}\right)}\left(\xi_T\left(X_i^t\right)-q\left(X_i^t, \theta^{\ast}\right)\right)\right|^2}\right)^{\frac{1}{2}} \notag\\
= & \frac{1}{\Omega_T} \sum_{t=1}^T \sum_{i=1}^{N_t}\left({\mathbb{E}\left|\frac{g(X_i^t)}{\xi_T\left(X_i^t\right) }\left(\xi_T\left(X_i^t\right)-q\left(X_i^t, \theta^{\ast}\right)\right)\right|^2}\right)^{\frac{1}{2}}, \label{term:TwoEstimatorsDiscrepancy}
\end{align}
where we use the notation $g(\boldsymbol{x}) = \frac{\pi(\boldsymbol{x})\psi(\boldsymbol{x})}{q(\boldsymbol{x}, \theta^{\ast})}$.
The analysis of the expectation term
\begin{equation}\label{term:SummandsInMAMIS-auxiliary}
\left({\mathbb{E}\left|\frac{g(X_i^t)}{\xi_T\left(X_i^t\right) }\left(\xi_T\left(X_i^t\right)-q\left(X_i^t, \theta^{\ast}\right)\right)\right|^2}\right)^{\frac{1}{2}}
\end{equation}
proceeds by 
\begin{align}
\left|\xi_T\left(X_i^t\right)-q\left(X_i^t, \theta^*\right)\right|
& = \left|\sum_{l=1}^T \frac{N_l}{\Omega_T} q\left(X_i^t, \widehat{\theta}_l\right)-q\left(X_i^t, \theta^{\ast}\right)\right| \notag\\
& =\left|\sum_{l=1}^T \frac{N_l}{\Omega_T}\left[q\left(X_i^t, \widehat{\theta}_l\right)-q\left(X_i^t, \theta^{\ast}\right)\right]\right| \notag\\
& \leq \sum_{l=1}^T \frac{N_l}{\Omega_T}\left|q\left(X_i^t, \widehat{\theta}_l\right)-q\left(X_i^t, \theta^{\ast}\right)\right| \notag\\
& \leq L\sum_{l=1}^T \frac{N_l}{\Omega_T} \lVert \widehat{\theta}_l-\theta^{\ast}\rVert , \notag
\end{align}
where the last inequality is derived from the Lipschitz assumption \cref{assumption:lipschitz}. Then the H\"older's inequality tells us \cref{term:SummandsInMAMIS-auxiliary} is not greater than
\begin{align}
& \Bigg(\mathbb{E}\Bigg[\Bigg(\Big|\frac{g(X_i^t)}{\xi_T(X_i^t) }\Big| \cdot L \sum_{l=1}^T \frac{N_l}{\Omega_T} \Big\|\widehat{\theta}_l-\theta^{\ast}\Big\|\Bigg)^2\Bigg]\Bigg)^{\frac{1}{2}} \notag\\
\leq & L \sum_{l=1}^T \frac{N_l}{\Omega_T} \Bigg(\mathbb{E}\left|\frac{g(X_i^t)}{\xi_T(X_i^t) }\right|^2\left\|\widehat{\theta}_l-\theta^*\right\|^2\Bigg)^{\frac{1}{2}} \notag\\
\leq & L \sum_{l=1}^{T} \frac{N_l}{\Omega_T}\left(\mathbb{E}\left|\frac{g(X_i^t)}{\xi_T(X_i^t) }\right|^{2u}\right)^{\frac{1}{2u}} \left(\mathbb{E}\left\|\widehat{\theta}_l-\theta^{\ast}\right\|^{2v}\right)^{\frac{1}{2v}}, \label{term:byHolder}
\end{align}
where $u=1+\frac{\eta}{2}, v = \frac{2+\eta}{\eta}$ and $\eta$ appears in \cref{assumption:moment}.
By the law of total expectation, 
\begin{align}
\mathbb{E}\left|\frac{g(X_i^t)}{\xi_T(X_i^t)}\right|^{2u} & \leq \mathbb{E}\left|\frac{g(X_i^t)}{m(X_i^t)}\right|^{2u} =\mathbb{E}\left[\mathbb{E}\left[\left|\frac{g(X_i^t)}{m(X_i^t)}\right|^{2+\eta} \Bigg| \widehat{\theta}_t\right]\right]. \label{term:moment conditional expectation}
\end{align}
Note that for any $\theta \in \Theta$, we have
\begin{align}
\mathbb{E}\left[\left.\left|\frac{g(X_i^t)}{m(X_i^t)}\right|^{2+\eta} \right\rvert\, \widehat{\theta}_t=\theta\right] \notag
& =\int_{\mathcal{X}} \left|\frac{g(\boldsymbol{x})}{m(\boldsymbol{x})} \right|^{2+\eta} q(\boldsymbol{x}, \theta) \mathrm{d} x \notag\\
& \leq \int_{\mathcal{X}} \left|\frac{g(\boldsymbol{x})}{m(\boldsymbol{x})}\right|^{2+\eta} \left(q\left(\boldsymbol{x}, \theta^{\ast}\right)+L \lVert\theta - \theta^{\ast} \rVert\right) \mathrm{d} \boldsymbol{x}, \notag
\end{align}
where for the last inequality, we use the Lipschitz assumption \cref{assumption:lipschitz}. 
Combining \cref{term:moment conditional expectation} and using the boundedness assumption of $q(\cdot,\theta^{\ast})$, we obtain 
\begin{align}
\mathbb{E}\Big|\frac{g(X_i^t)}{\xi_T(X_i^t)}\Big|^{2u} 
\leq \left({\lVert q(\cdot,\theta^{\ast}) \rVert}_{\infty}+L\mathbb{E} \lVert\widehat{\theta}_t - \theta^{\ast} \rVert\right)\int_{\mathcal{X}} \left|\frac{g(\boldsymbol{x})}{m(\boldsymbol{x})}\right|^{2+\eta}  \mathrm{d}\boldsymbol{x}. \notag
\end{align}
Since the right hand side of inequality \cref{term:ThetaError} is bounded as a function with respect to $N_t$, it follows that we can bound $\mathbb{E}\lVert\widehat{\theta}_t - \theta^{\ast} \rVert$ by a constant which is independent of $t$. Combining‌ the assumption \cref{assumption:moment}, we obtain $\mathbb{E}\Big|\frac{g(X_i^t)}{\xi_T(X_i^t)}\Big|^{2u}$ can be bounded by a constant which is independent of $t$.

Consequently, from \cref{term:byHolder} and \cref{thm:ThetaError}, we can bound \cref{term:SummandsInMAMIS-auxiliary} as follows.
\begin{align}
& \left({\mathbb{E}\left|\frac{g(X_i^t)}{\xi_T\left(X_i^t\right) }\left(\xi_T\left(X_i^t\right)-q\left(X_i^t, \theta^{\ast}\right)\right)\right|^2}\right)^{\frac{1}{2}} \notag\\
\lesssim & \sum_{l=1}^T \frac{N_l}{\Omega_T}\left[N_l^{-1}(\log N_l)^{\frac{3d}{2}-1}\exp{\{M(8vd\log N_l)^{\frac{k}{2}}\}}\right] \notag\\
\lesssim & \frac{1}{\Omega_T}\sum_{l=1}^T N_l^{\epsilon} \lesssim  (\frac{\Omega_T}{T})^{-1+\epsilon} = \bar{N}_T^{-1+\epsilon}, \forall \epsilon>0, \notag
\end{align}
where we apply the Jensen's inequality to the concave function $f(x) = x^\epsilon$ in the last inequality. As a result, we can also bound the right hand side of \cref{term:TwoEstimatorsDiscrepancy} by $\mathcal{O}(\bar{N}_T^{-1+\epsilon})$.

Based on the preceding discussion, we can obtain the following result for the error term $|\widehat{\Pi}_{T}^{\mathrm{MAMIS}}(\psi) - \widehat{\Pi}_{T}^{\ast}(\psi)|$.
\begin{theorem}\label{thm:AuxiliaryEstimateError}
Under \cref{assumption:all}, for any $\epsilon>0$, we have
\begin{equation}
 \left(\mathbb{E}\left|\widehat{\Pi}_T^{\mathrm{MAMIS}}(\psi)-\widehat{\Pi}_T^*(\psi)\right|^2\right)^{\frac{1}{2}} = \mathcal{O}(\bar{N}_T^{-1+\epsilon}),  \notag  
\end{equation}
where the constant in the big-$\mathcal{O}$ bound depends on $M,B,k,d,D,L,\eta,\epsilon$.
\end{theorem}

Now we can prove \cref{thm:MAMIS-RQMC}.
\begin{proof}[Proof of \cref{thm:MAMIS-RQMC}]
Combining \cref{cor:Auxiliary-discrepancy}, \cref{thm:AuxiliaryEstimateError}, we obtain
\begin{align}
\Big(\mathbb{E}\Big|\widehat{\Pi}_T^{\mathrm{MAMIS}}(\psi) - I\Big|^2\Big)^{\frac{1}{2}} 
& \leq  \Big(\mathbb{E}\Big|\widehat{\Pi}_T^{\mathrm{MAMIS}}(\psi) - \widehat{\Pi}_T^*(\psi)\Big|^2\Big)^{\frac{1}{2}} +  \Big(\mathbb{E}\Big|\widehat{\Pi}_T^*(\psi) - I\Big|^2\Big)^{\frac{1}{2}}\notag\\
& = \mathcal{O}(\bar{N}_T^{-1+\epsilon}), \notag
\end{align}
where the constant in the big-$\mathcal{O}$ bound depends on $M,B,k,d,D,L,\eta,\epsilon$.
\end{proof}

\section{Convergence Rate of the MC-based MAMIS Estimator}
\label{sec:MC_convergence}
It is worth noting that Marin et al.\cite{Marin2019} only provided the consistency of the MC-based MAMIS, but did not give the convergence rate with respect to any quantity. It is observed that in our analytical framework under the RQMC setting, all error analyses are based on \cref{thm:Lq-convergence}. We can obtain a parallel version of \cref{thm:Lq-convergence} for the MC sampling case by using the Marcinkiewicz–Zygmund inequality (see \cite{Ferger2014M-Zineq}), thereby establishing the convergence order results of the MC-based MAMIS estimator. In this section, $X$ still denotes a random variable following the standard normal distribution.

\begin{lemma}\label{lemma:Lq-convergence-MC}
Given $q \geq 2$, then for function $f$ satistied $A := \E[|f(X)|^q] < \infty$ and the independent and identically distributed‌ ($i.i.d.$) point set $\left\{y_{1},\ldots, y_{n}\right\}$ following $\mathcal{U}[0,1)^d$, we have
\begin{equation}
 \E\left[\left\lvert \widehat{I}_{n}(f)-\E[f(X)]\right\lvert^{q}\right] = \mathcal{O}(n^{-\frac{q}{2}}),\notag 
\end{equation}
where $\widehat{I}_{n}(f)$ is given in \cref{term:RQMCEstimator_ToIntroducePQMC} and the constant in the big-$\mathcal{O}$ bound depends on $q,A$.
\end{lemma}
\begin{proof}
Let $Y_j = f \circ \Phi^{-1}(y_j) - \mathbb{E}[f(X)]$ for $j=1,2,\ldots,n$. Then $Y_1,\ldots,Y_n$ are $i.i.d.$ random variables with zero mean. Using the Jensen's inequality, we have
\begin{align}
\E[|Y_j|^q] &\leq 2^{q-1}\E\Big[|f\circ\Phi^{-1}(y_j)|^q+|\E[f(X)]|^q\Big] \notag\\
& = 2^{q-1}(\E[|f(X)|^q]+|\E[f(X)]|^q) \notag\\
& \leq 2^q\E[|f(X)|^q]  = (2A)^q. \notag
\end{align}
Hence we derive
\begin{align}
\E\Big[\left\lvert \widehat{I}_{n}(f)-\E[f(X)]\right\lvert^{q}\Big]
=& \frac{1}{n^q}\E\Big[\Big\lvert\sum_{j=1}^n Y_j\Big\rvert^q \Big] \notag \\
\leq & \frac{1}{n^q}C_qn^{\frac{q}{2}-1} \sum_{j=1}^n\E\Big[\Big\lvert  Y_j\Big\rvert^q \Big] 
\leq (2A)^qC_qn^{-\frac{q}{2}}, \label{term:usingM-Zineq}
\end{align}
where we use the Marcinkiewicz–Zygmund inequality in \cref{term:usingM-Zineq}. 
\end{proof}

By using this lemma and following an analysis process similar to that in the RQMC case, we can obtain the convergence order result of the MC-based MAMIS estimator.
\begin{theorem}
Under \cref{assumption:all}, if there exists a positive constant $A<\infty$, such that $\E[|\frac{\pi \psi}{q(\cdot,\theta^{\ast})}\circ T_{\theta}(X)|^2]$, $\E[|\frac{\pi h^{(j)}}{q(\cdot,\theta^{\ast})}\circ T_{\theta}(X)|^{\frac{4+2\eta}{\eta}}] \leq A$ $(\forall \theta \in \Theta, j = 1,\ldots,d)$, then the MC-based MAMIS estimator has an RMSE as
\begin{align}
 \Big(\mathbb{E}\Big|\widehat{\Pi}_T^{\mathrm{MAMIS}}(\psi) - I\Big|^2\Big)^{\frac{1}{2}}
= \mathcal{O}(\bar{N}_T^{-1/2}),\notag
\end{align}
where the constant in the big-$\mathcal{O}$ bound depends on $L,\eta,A$.
\end{theorem}

\section{Experimental results}
\label{sec:experiments}
In our numerical experiments, we use the scrambled Sobol' sequence in the RQMC methods and choose the proposals from the normal distribution family or Student's $t$ distribution family. Firstly, we will demonstrate the validity of  \cref{algorithm:MAMIS-unnormalized} with RQMC sampling through a mixture of Gaussians example. Subsequently, we propose \cref{algorithm:MAMIS-normalized}, which includes self-normalized weights to address the issue of weights potentially being rounded to zero due to computational precision limitations. This algorithm can also be applied when the normalizing constant of the target density is unknown. Additionally, we illustrate the effectiveness of \cref{algorithm:MAMIS-normalized} with RQMC sampling through examples involving mixture of Gaussians, banana-shaped model and Bayesian Logistic regression.
\subsection{Mixture of Gaussians}\label{subsec:MixtureofGaussians}
A mixture of Gaussians, also known as a Gaussian Mixture Model (GMM), is a probabilistic model that assumes all the data points are generated from a mixture of a finite number of Gaussian distributions. Formally, we denote $\mathcal{N}(\boldsymbol{x} \mid \mu, \Sigma)$ as the density of Gaussian distribution with mean vector $\mu$ and covariance matrix $\Sigma$. We consider the target density as
\begin{equation}\label{term:MixGaussians}
\pi(\boldsymbol{x}) = \frac{1}{K}\sum_{i=1}^{K} \mathcal{N}(\boldsymbol{x} \mid \mu_i, \Sigma_i).
\end{equation}
\subsubsection{The case that all $\Sigma_i$ are the same}\label{subsubsec:Toymodel}
We consider a 20-dimensional problem firstly. Denote $\boldsymbol{\underline{1}} = (1,1,\ldots,1)^{\top} $ as the vector in $\mathbb{R}^d$ with all components set to $1$ and let $\mu_1 = \boldsymbol{\underline{1}}, \mu_2 = \boldsymbol{0}, \mu_3 = -\boldsymbol{\underline{1}}$, where $\boldsymbol{0}$ is the zero vector in $\mathbb{R}^d$. We consider the case that the Gaussian distributions used for the mixture all share the same covariance matrix, that is
\begin{equation}
\pi(\boldsymbol{x}) =  \frac{1}{3} \sum_{i=1}^{3} \mathcal{N}(\boldsymbol{x} \mid \mu_i, \Sigma), \label{term:ToymodelTarget}  
\end{equation}
where $\Sigma $ denotes a $ d \times d $ matrix with all diagonal elements equal to $d$ and all remaining elements equal to 1. Now we use the MAMIS estimator to estimate the integral
\begin{equation}
\int_{\mathbb{R}^d} \psi(x)\pi(x) \mathrm{d}\boldsymbol{x} \notag
\end{equation}
for $d=20$ and $\psi(\boldsymbol{x}) = x_1^2$, where $\boldsymbol{x} = (x_1,\ldots,x_d)^{\top}$. This integral can be calculated exactly as 
\begin{align}
\frac{1}{3}\sum_{i=1}^{3}\int_{\mathbb{R}^d} x_1^2 \mathcal{N}(\boldsymbol{x} \mid \mu_i, \Sigma) \mathrm{d}\boldsymbol{x}
= d + \frac{2}{3}. \notag
\end{align}
We choose the proposals from normal distribution family $\mathcal{N}(\theta, \Sigma)$, that is, setting $q(\boldsymbol{x},\theta) = \mathcal{N}(\boldsymbol{x} \mid \theta, \Sigma)$ where $\Sigma$ is the same as in \cref{term:ToymodelTarget}. We propose that $\theta^{\ast} = \mathbb{E}_{\pi}[X]$, hence the $h$ in \cref{term:ThetastarForm} is $h(\boldsymbol{x}) = \boldsymbol{x}$. We set initial value of $\theta$ to be $(0.1,0.1,\ldots,0.1)^{\top}$ and $T = 64$ with the same number of points being used at each step, thus the number of points used at each step is equal to $\bar{N}_T$. We compute the RMSEs based on $J = 50$ independent repetitions and increase $\bar{N}_T$ from $2^{10}$ to $2^{16}$. 

We compare the MAMIS with the optimal drift IS (ODIS) method (see \cite{He2023}). ODIS is a common IS method that chooses the proposal from the normal distribution family, where the mean of proposal density is chosen to match the mode of the integrand. Formally, we use $\mathcal{N}(\boldsymbol{x}|\mu^{\ast},\Sigma)$ as the density of proposal, where $\Sigma$ is the same as in \cref{term:MixGaussians} and $\mu^{\ast} = \arg \max \limits_{\boldsymbol{x} \in \mathbb{R}^d} \psi(\boldsymbol{x})\pi(\boldsymbol{x})$ for the integrand $\psi$. To ensure fairness, we use the same total sample size $\Omega_T$ in ODIS as in MAMIS.

From \cref{fig:Toy model}, we can see the MC-based MAMIS converges at a rate $\mathcal{O}(\bar{N}_T^{-1/2})$ while the RQMC based MAMIS reaches $\mathcal{O}(\bar{N}_T^{-1+\epsilon})$. Moreover, the RQMC-based MAMIS outperforms all other methods. 
\begin{figure}[htbp]
    \centering
    \includegraphics[width=0.8\linewidth]{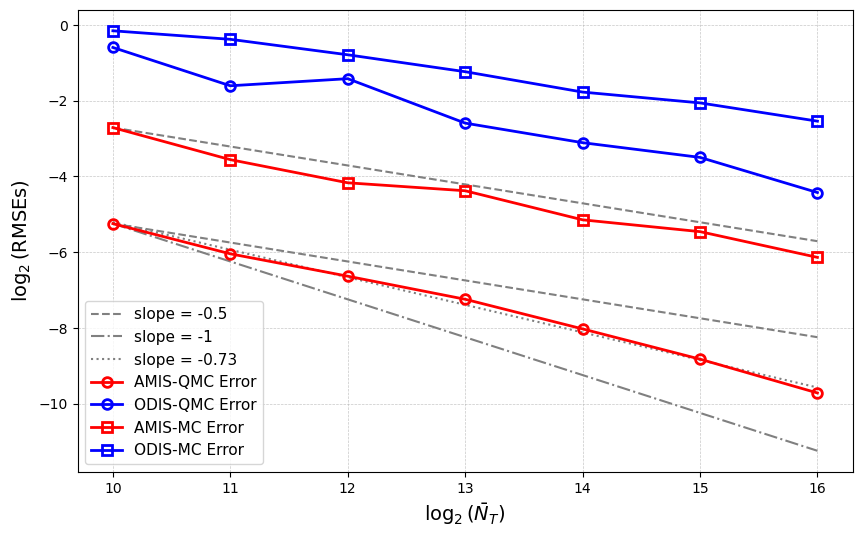}
    \caption{Log-RMSEs for the mixture of three Gaussians example where all $\Sigma_i$ are the same and the dimension of this example is $20$. The RMSEs are computed based on 50 repetitions.}
    \label{fig:Toy model}
\end{figure}

Although this example demonstrates the effectiveness of \cref{algorithm:MAMIS-unnormalized}, in more complex cases, it may occur that all calculated weights are very small (even though some are relatively larger). Due to the limitations of machine precision, all weights may be truncated to zero, leading to an invalid estimate. And in some problems, such as Bayesian computing, the normalizing constant of the target density is unknown. To overcome these challenges, we consider the inclusion of weight self-normalization in the MAMIS approach as the following \cref{algorithm:MAMIS-normalized}. The effectiveness of this algorithm will be shown in the subsequent experiments.
\begin{algorithm}[htbp]
\caption{Self-Normalized MAMIS}
\begin{algorithmic}\label{algorithm:MAMIS-normalized}
\STATE \textbf{Input:} an initial parameter $\widehat{\theta}_1$, and sample sizes $N_1, \ldots, N_T$.
\FOR{$t = 1 \to T$}
    \FOR{$i = 1 \to N_t$}
        \STATE \textbf{draw} $X_i^t$ from $Q(\widehat{\theta}_t)$
        \STATE \textbf{compute} $W_i^t = \frac{\pi(X_i^t)}{q(X_i^t, \widehat{\theta}_t)}$
        \STATE \textbf{compute} $\omega_i^t = \frac{W_i^t}{\sum_{i=1}^{N_t} W_i^t}$
    \ENDFOR
    \STATE \textbf{compute} $\widehat{\theta}_{t+1} = N_t^{-1} \sum_{i=1}^{N_t} \omega_i^t h(X_i^t)$
\ENDFOR
\STATE \textbf{set} $\Omega_T = N_1 + \cdots + N_T$
\FOR{$t = 1 \to T$}
    \FOR{$i = 1 \to N_t$}
        \STATE \textbf{update} $W_i^t = \frac{\pi(X_i^t)}{{\Omega_{T}^{-1} \sum_{l=1}^T N_l q(X_i^t, \widehat{\theta}_l)}}$
        \STATE \textbf{update} $\omega_i^t = \frac{W_i^t}{\sum_{i=1}^{N_t} W_i^t}$
    \ENDFOR
\ENDFOR
\RETURN $(X_1^1, \omega_1^1), \ldots, (X_{N_1}^1, \omega_{N_1}^1), \ldots, (X_1^T, \omega_1^T), \ldots (X_{N_T}^T, \omega_{N_T}^T).$
\end{algorithmic}
\end{algorithm}

\subsubsection{The case that $\Sigma_i$ are different}\label{subsubsec:MixGaussiansSigmaCanbeDifferent}
We consider a two-dimensional mixture of five normals example from \cite{Huang2022PQMC} with proper scaling. We set
\begin{equation}
\pi(\boldsymbol{x}) = \frac{1}{5} \sum_{i=1}^{5} \mathcal{N}(\boldsymbol{x} \mid \mu_i, \Sigma_i), \notag
\end{equation}
where
$\mu_1 = [1, 1]^{\top}$,
$\mu_2 = [2, 3.6]^{\top}$,
$\mu_3 = [3.3, 2.8]^{\top}$,
$\mu_4 = [1.1, 2.9]^{\top}$,
$\mu_5 = [3.4, 0.6]^{\top}$,
and
$\Sigma_1 = \frac{1}{40^2} [2, 0.6; 0.6, 1]^{\top}$,
$\Sigma_2 = \frac{1}{40^2} [2, -0.4; -0.4, 2]^{\top}$,
$\Sigma_3 = \frac{1}{40^2} [2, 0.8; 0.8, 2]^{\top}$,
$\Sigma_4 = \frac{1}{40^2} [3, 0; 0, 0.5]^{\top}$,
$\Sigma_5 = \frac{1}{40^2} [2, -0.1; -0.1, 2]^{\top}$.

In \cref{subsubsec:Toymodel} , since the Gaussian distributions used for mixing have the same covariance matrix $\Sigma$, our proposals are chosen from a normal distribution family with a fixed covariance matrix $\Sigma$. However, in more complex cases, we will take into account the impact of covariance matrices. Thus, we consider parameters 
\begin{equation}
\theta = (\theta_1,\ldots,\theta_d,\theta_{d+1},\ldots,\theta_{d+d^2})^{\top} \notag
\end{equation}
 in $\Theta \subseteq \mathbb{R}^{d+d^2}$ and denote $\mathcal{N}(\theta) := \mathcal{N}(\mu_{\theta}, \Sigma_{\theta})$, where $\mu_{\theta} = (\theta_1, \ldots, \theta_d)^{\top}$ and
\begin{equation}
\Sigma_{\theta} = \begin{pmatrix}
\theta_{d+1} & \theta_{d+2} & \ldots & \theta_{2d} \\
\theta_{2d+1} & \theta_{2d+2} & \ldots & \theta_{3d} \\
\vdots & \vdots & \ddots & \vdots \\
\theta_{d^2+1} & \theta_{d^2+2} & \ldots & \theta_{d^2+d}
\end{pmatrix}. \notag    
\end{equation}

We choose proposals from $\{\mathcal{N}(\theta)\}$ and propose 
\begin{equation}
\theta^{\ast} = \begin{pmatrix}
\mathbb{E}_{\pi}[X]\\
\text{vec}(\text{Cov}_{\pi}(X))  
\end{pmatrix},\notag
\end{equation}
where $\text{vec}(A)$ means the vectorization of matrix $A = (a_{ij})_{d\times d}$, that is, transforming $A$ into a vector
\begin{equation}
\text{vec}(A) = (a_{11},a_{12},\ldots,a_{1d},a_{21},\ldots,a_{2d},\ldots,a_{d1},\ldots,a_{dd})^{\top}. \notag
\end{equation}
Hence the $h$ in \cref{term:ThetastarForm} can be written as
\begin{equation}
h(\boldsymbol{x}) = \begin{pmatrix}
\boldsymbol{x}\\
\text{vec}((\boldsymbol{x} - \mathbb{E}_{\pi}[X])(\boldsymbol{x} - \mathbb{E}_{\pi}[X])^{\top})  
\end{pmatrix}.\label{term:FunctionhOfMixGau}
\end{equation}

Since $\mathbb{E}_{\pi}[X]$ may be unknown, we initially use a small number of samples to run the algorithm and provide an approximate value for $\mathbb{E}_{\pi}[X]$ under 
\begin{equation}
h(\boldsymbol{x}) = \begin{pmatrix}
\boldsymbol{x}\\
\text{vec}(\boldsymbol{x}\boldsymbol{x}^{\top})  
\end{pmatrix}.\notag
\end{equation}
We will then use this value as $\mathbb{E}_{\pi}[X]$ in \cref{term:FunctionhOfMixGau} to run the algorithm with more samples to compute the integral $\int_{\mathcal{X}} \psi(\boldsymbol{x})\pi(\boldsymbol{x}) \mathrm{d}\boldsymbol{x}$, while this value can also serve as a good initial value of the first $d$ components of $\theta$. Here we use $\psi(\boldsymbol{x}) = \boldsymbol{x}$ (see \cref{remark:PsiCanBeVector}) to 
test the algorithm. Firstly, we set initial value of $\theta$
\begin{equation}
\widehat{\theta}_1 = \begin{pmatrix}
\boldsymbol{0}\\
\text{vec}(I_d)  
\end{pmatrix} \notag
\end{equation} and $T=32,J=10$ and use $2^4$ points at each step to obtain an approximate value for $\mathbb{E}_{\pi}[X]$, which is denoted as $\widehat{\boldsymbol{\mu}}$. Then we set 
\begin{equation}
\widehat{\theta}_1 = \begin{pmatrix}
\widehat{\boldsymbol{\mu}}\\
\text{vec}(I_d)  
\end{pmatrix} \notag
\end{equation} and $T=64,J=50$ and use the same number $\bar{N}_T$ of points at each step to estimate the target mean using \cref{algorithm:MAMIS-normalized}, where $\bar{N}_T$ increases from $2^{11}$ to $2^{17}$. The exact value of $\int_{\mathcal{X}} \psi(\boldsymbol{x})\pi(\boldsymbol{x}) \mathrm{d}\boldsymbol{x}$ is $(2.16,2.14)^T$ and we compute the RMSEs as in \cref{fig:MixtureGaussians}. We observe that the convergence rate of RQMC-based MAMIS can reach $\mathcal{O}(\bar{N}_T^{-3/2+\epsilon})$, which may be attributed to the excellent properties of scrambled nets; see \cite{Owen2008Scrambled}.
\begin{figure}[htbp]
    \centering
    \includegraphics[width=0.8\linewidth]{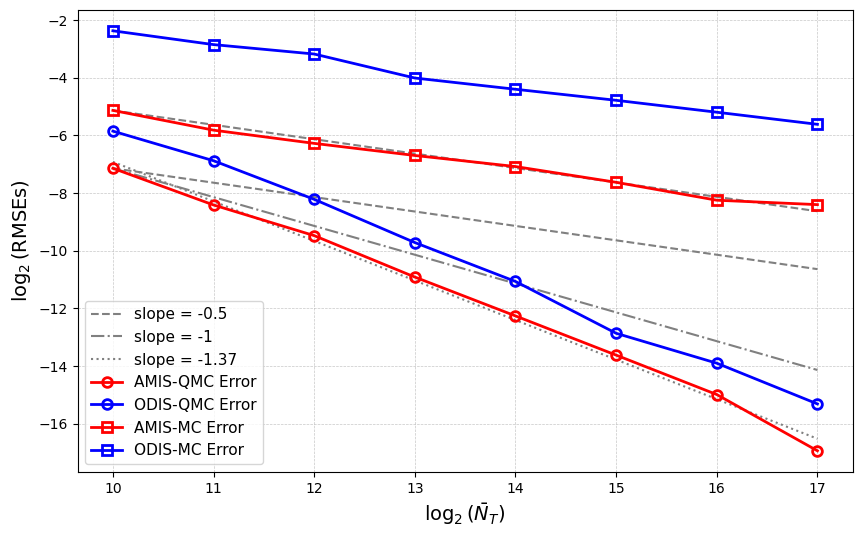}
    \caption{Log-RMSEs for the two-dimensional mixture of five Gaussians example where $\Sigma_i$ are different. The RMSEs are computed based on 50 repetitions.}
    \label{fig:MixtureGaussians}
\end{figure}

\subsection{Bayesian Logistic regression}
Bayesian Logistic regression is a probabilistic approach to modeling binary outcomes using Logistic regression within a Bayesian framework. In this model, predictors $X = ({X}_1, \ldots, {X}_n)^{\top}$ are related to the binary response ${Y} = (Y_1, \ldots, Y_n)^{\top}$ through a vector of unknown parameters $\boldsymbol{z} \in \mathbb{R}^d$. The likelihood function for the data is then expressed as

\begin{equation}
\ell({\boldsymbol{z}}) = \prod_{i=1}^n \left( \frac{\exp({X}_i^{\top} {\boldsymbol{z}})}{1 + \exp({X}_i^{\top} {\boldsymbol{z}})} \right)^{Y_i} \left( \frac{1}{1 + \exp({X}_i^{\top} {\boldsymbol{z}})} \right)^{1-Y_i}.\notag
\end{equation}

Under the Bayesian framework, the parameter ${\boldsymbol{z}}$ is modeled as a random vector with a prior distribution assumed to be a standard normal distribution $\mathcal{N}({0}, {I}_d)$. The posterior distribution of ${z}$ given ${Y}$ is then given by $ \pi( {z}) \propto \mathcal{N}(\boldsymbol{z}|\boldsymbol{0},I_d) \ell( {\boldsymbol{z}})$.

Owen \cite{owenqmc} suggested using the multivariate Student's $t$ distribution $t_\nu(\mu,\Sigma)$ as proposal in Bayesian Logistic regression and He et al. \cite{He2023} studied choosing proposal from $t_\nu(\mu,\Sigma)$ by ODIS and Laplace IS (LapIS). The LapIS uses $\mathcal{N}(\boldsymbol{x}|\mu^{\ast},\Sigma^{\ast})$ as the density of proposal, where $\mu^{\ast}$ is the same as in ODIS, and $\Sigma^{\ast} = -\nabla^2H(\mu^{\ast})$ with $H(\boldsymbol{x}) = \log(\psi(\boldsymbol{x})\pi(\boldsymbol{x}))$. He et al. \cite{He2023} used $t_\nu(\mu^{\ast},\Sigma^{\ast})$ as the proposal with the same $\mu^{\ast},\Sigma^{\ast}$ as in LapIS.

Hence we employ a multivariate Student's $t$ distribution $t_\nu(\mu,\Sigma)$ as the proposal family, where $\nu$ is the degree of freedom. We set $\nu=2$ and use the same method as in \cref{subsubsec:Toymodel} to update $\mu$ and $\Sigma$. We use the first 30 entries of pima dataset \cite{Ripley1996} to compute likelihood function and estimate $\int_{\mathbb{R}^d} \lVert\boldsymbol{z}\rVert^2 \mathrm{d}\boldsymbol{z}$ where $d=9$ and $\lVert\boldsymbol{z}\rVert^2 = \sum \limits_{i=1}^d z_i^2$. We use the same method to choose the initial value of mean  as in \cref{subsubsec:MixGaussiansSigmaCanbeDifferent}, and record the covariance matrices updated at the end of each experiment in this process, taking their average as the initial value of the covariance matrix. We set $T=64,J=50$ and use the same number $\bar{N}_T = 2^{17}$ of RQMC points at each step to compute a value as the true value. Then we increase $\bar{N}_T$ from $2^{9}$ to $2^{15}$ to calculate the RMSEs as in \cref{fig:BLR}. 
We compare MAMIS with  ODIS ($\mu = \mu^{\ast}, \Sigma=I_d$) and  LapIS. We can observe that the RQMC-based MAMIS estimator converges at $\mathcal{O}((\bar{N}_T)^{-0.64})$ better than the MC-based MAMIS estimator. Using $t$ distribution as the proposal family, the RQMC-based MAMIS and the RQMC-based LapIS performs similarly while the latter performs best in \cite{He2023}. 

\begin{figure}[htbp]
    \centering
    \includegraphics[width=0.8\linewidth]{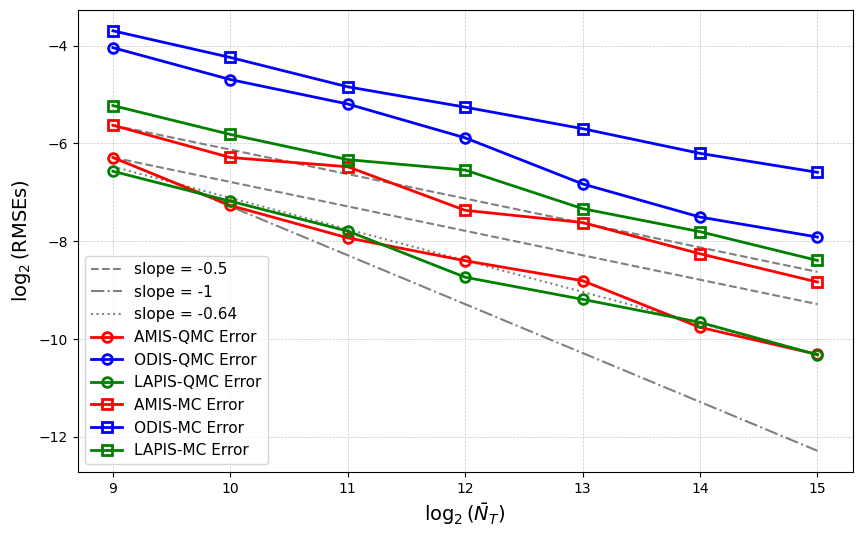}
    \caption{Log-RMSEs for the Bayesian Logistic regression example. The RMSEs are computed based on 50 repetitions.}
    \label{fig:BLR}
\end{figure}

\subsection{Banana-shaped target example}\label{subsec:Banana Model}
The banana-shaped target “can be calibrated to become extremely challenging” \cite{Cornuet2012}. We consider a two-dimensional banana-shaped target example \cite{Martino2017LAIS}. The target density is
\begin{equation}
\pi(\boldsymbol{x}) \propto \exp \left\{ -\frac{1}{2\eta_1^2} \left( 4 - bx_1 - x_2^2 \right)^2  - \frac{x_2^2}{2\eta_2^2} \right\}.
\end{equation} 
We set $\eta_1=3,\eta_2=2$ and $b=10$, then the target mean $\int_{\mathcal{X}} \boldsymbol{x} \mathrm{d}\boldsymbol{x}$ can be known as zero.
We set $T = 64$ and other settings are the same as in \cref{subsubsec:MixGaussiansSigmaCanbeDifferent} to compute the RMSEs between the target mean and its MAMIS estimator as follows. 
We compare the MAMIS with ODIS and  LapIS. We can observe that the RQMC-based MAMIS estimator converges at $\mathcal{O}((\bar{N}_T)^{-0.84})$, which is much better than the MC-based MAMIS. In contrast, the RQMC-based ODIS estimator and the RQMC-based LapIS estimator have much larger errors compared to the RQMC-based MAMIS estimator.

\begin{figure}[htbp]
    \centering
    \includegraphics[width=0.8\linewidth]{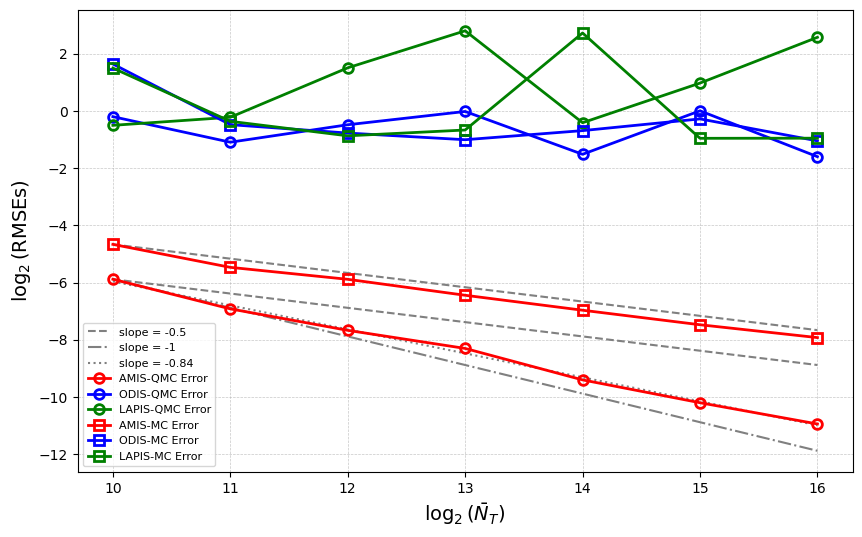}
    \caption{Log-RMSEs for the two-dimensional Banana-shaped target example. The RMSEs are computed based on 50 repetitions.}
    \label{fig:Banana model}
\end{figure}

\section{Conclusions}
\label{sec:conclusions}
In this study, we have explored the convergence properties of the MAMIS algorithm under the RQMC setting. By integrating RQMC methods into the MAMIS framework, we have demonstrated significant increase in the convergence rates compared to the traditional MC-based approaches. Specifically, we have shown that the RMSE of the RQMC-based MAMIS estimator can achieve an error rate of $\mathcal{O}(\bar{N}_T^{-1+\epsilon})$ for an arbitrarily small $\epsilon > 0$, where $T$ denotes the number of iteration and $\bar{N}_T$ represents the average number of samples used in each step over $T$ iterations.

Our theoretical analysis and numerical experiments highlight several key findings. Firstly, the incorporation of RQMC methods into the MAMIS significantly enhances the efficiency of the sampling process. This is attributed to the superior convergence properties of QMC methods compared to the traditional MC sampling. Secondly, we established higher-order $L^q$ $(q\geq2) $ error bounds for RQMC estimators. Existing results on error estimates for RQMC estimators often focus only on $L^1$ and $L^2$ errors. This limitation poses difficulties when we attempt to analyze errors using Hölder's inequality under the RQMC setting. Therefore, our results on higher-order error estimates provide a more powerful tool for conducting error analysis under the RQMC framework. Finally, within our theoretical framework, we provided the convergence rate results under the MC setting, thereby complementing the work of Marin et al. \cite{Marin2019}, which was limited to consistency results. The numerical results validate the effectiveness of our method in achieving faster convergence rates and improved estimation accuracy.

In our theoretical analysis, we retained the assumption \cref{term:ThetastarForm} from Marin et al. \cite{Marin2019}.  Although we used a moment-matching criterion in our experiments, which satisfies this assumption, there are some other criterion under which this assumption may not hold. However, it is important to note that this assumption is only used to design the algorithm for optimizing the parameters. Within our error analysis framework, we only need to know the parameter error, and the assumption itself is not required. Therefore, if we introduce other stochastic optimization algorithms that do not require this assumption to optimize the parameters, we might be able to weaken this assumption. In our experiments, we proposed a self-normalized MAMIS algorithm but did not conduct an error analysis for it. We leave these issues for future work.

\appendix
\section{Proof of \cref{thm:Lq-convergence}} \label{appendix:Proof of Lq-convergence}

Before the proof of \cref{thm:Lq-convergence}, we first conduct some analysis and prove two lemmas. 

By Jensen's inequality, we have 
\begin{equation}
\left( \frac{a + b + c}{3}\right)^q \leq \frac{a^q + b^q +c^q}{3}(a,b,c > 0),
\notag
\end{equation}
implying that 
\begin{align}
3^{1-q}\left|\widehat{I}_{n}(f)-\E[f(X)]\right|^{q} \leq & \big|\widehat{I}_{n}(f)-\widehat{I}_{n}(f\circ P_R)\big|^{q} + \big|\widehat{I}_{n}(f\circ P_R)-\E\left[f \circ P_{R}(X)\right]\big|^{q} \notag\\ 
& + {\big|\E\left[f \circ P_{R}(X)\right]-\E[f(X)]\big|}^{q}.\notag
\end{align}
Taking the expectation on both sides of this inequality, we obtain 
\begin{align}
3^{1-q}\E\left|\widehat{I}_{n}(f)-\E[f(X)]\right|^{q} \leq & \E\big|\widehat{I}_{n}(f)-\widehat{I}_{n}(f\circ P_R)\big|^{q} + \E\big|\widehat{I}_{n}(f\circ P_R)-\E\left[f \circ P_{R}(X)\right]\big|^{q} \notag\\ 
& + {\big|\E\left[f \circ P_{R}(X)\right]-\E[f(X)]\big|}^{q}.\notag\\
=: & E_1 + E_2 +E_3 \label{term:ErrorDecompositionInAppendix}
\end{align}

$E_1$ and $E_3$ are related to the following lemma.
\begin{lemma}\label{lemma:ProjectionError}
For any $f \in G_e(M,B,k)$ with $0<k<2$, we have
\begin{align}
\E\left|f(X)-f \circ P_{R}(X)\right|^{q} \leq C_0(R-1)^{d-2} \exp \left\{-\frac{1}{4}(R-1)^{2}\right\}, \notag
\end{align}
where $C_0= 2^q(\frac{\pi}{2})^{\frac{d}{2}-1} B^{q} \exp \left\{\frac{2-k}{2} qM(2 q k M)^{\frac{k}{2-k}}\right\}(d+1)!!$ and the projection radius $R>\sqrt{2}+1$.
\end{lemma}
\begin{proof}
Note that
\begin{equation}
\E\left|f(X)-f \circ P_{R}(X)\right|^{q}
\leq (2 \pi)^{-\frac{d}{2}} \int_{|\boldsymbol{x}| \geqslant R-1}\left(|f(\boldsymbol{x})|+\left|f \circ P_{R}(\boldsymbol{x})\right|\right)^{q} e^{-|\boldsymbol{x}|^{2} / 2} \mathrm{d}\boldsymbol{x}.
\label{term:projerror}
\end{equation}
Then for any $f \in G_{e}(M, B, K)$, the right hand side of \cref{term:projerror}
\begin{align}
& \leq(2 \pi)^{-\frac{d}{2}} \int_{|\boldsymbol{x}| \geqslant R-1}\left(2 B \exp \left\{M|\boldsymbol{x}|^{k}\right\}\right)^{q} e^{-|\boldsymbol{x}|^{2} / 2} d \boldsymbol{x} \notag\\
& =\frac{2^q B^q}{(2 \pi)^{\frac{d}{2}}} \int_{0}^{2 \pi} \ldots \int_{0}^{\pi} \prod_{j=1}^{d-2}\left|\sin \psi_{j}\right|^{d-1-j} \int_{R-1}^{\infty} r^{d-1}\exp \left\{q M r^{k}\right\} e^{-r^{2} / 2} \mathrm{d}r \mathrm{d}\psi_{1:(d-1)}, \notag
\end{align}
where we use the polar coordinates transformation in the last equality.
From the fact that $|\sin\psi| \leq 1$, we obtain the right hand side of \cref{term:projerror} is bounded by
\begin{equation}
2^q B^q(\frac{\pi}{2})^{\frac{d}{2}-1} \int_{R-1}^{\infty} r^{d-1} \exp \left\{q M r^{k}\right\} e^{-r^{2} / 2} \mathrm{d}r.
\label{Geterm}
\end{equation}
By Young's inequality, for any $ \epsilon>0$, 
\begin{equation}
r^{k} \leq \frac{\left(\epsilon r^{k}\right)^{\frac{2}{k}}}{\frac{2}{k}}+\frac{\left(\frac{1}{\epsilon}\right)^{\frac{2}{2-k}}}{\frac{2}{2-k}}=\frac{k}{2} \epsilon^{\frac{2}{k}} r^{2}+\frac{2-k}{2} \left(\frac{1}{\epsilon}\right)^{\frac{2}{2-k}}, \notag
\end{equation}
let $ \epsilon^{\frac{2}{k}}=\frac{1}{2 q kM}$, we get
\begin{equation}
 q M r^{k} \leq \frac{1}{4} r^{2}+\frac{2-k}{2}  q M (2 q k M)^{\frac{k}{2-k}}, \notag
\end{equation}
then \cref{Geterm} can be bounded as follows.
\begin{align}
& 2^q B^q(\frac{\pi}{2})^{\frac{d}{2}-1} \int_{R-1}^{\infty} r^{d-1} \exp \left\{q M r^{k}\right\} e^{-r^{2} / 2} \mathrm{d}r \notag\\
\leq & 2^q(\frac{\pi}{2})^{\frac{d}{2}-1} \exp \left\{\frac{2-k}{2} qM (2 q k M)^{\frac{k}{2-k}}\right\} B^{q} \int_{R-1}^{\infty} r^{d-1} e^{-\frac{r^{2}}{4}} d r \notag\\
\leq & 2^q(\frac{\pi}{2})^{\frac{d}{2}-1} B^{q} \exp \left\{\frac{2-k}{2} qM(2 q k M)^{\frac{k}{2-k}}\right\}(d+1)!!(R-1)^{d-2} \exp \left\{-\frac{1}{4}(R-1)^{2}\right\}, \notag
\end{align}
where we use the lemma A.2 in \cite{Ouyang2024} for the second inequality.
\end{proof}

Note that, for $E_1$ in \cref{term:ErrorDecompositionInAppendix}, we have
\begin{align}
E_1
= &\E\left|\widehat{I}_{n}(f)-\widehat{I}_{n}(f\circ P_R)\right|^{q} \notag\\
= & \E\left[\left|\frac{1}{n} \sum_{j=1}^{n} f \circ \Phi^{-1}\left(y_{j}\right)-\frac{1}{n} \sum_{j=1}^{n} f \circ P_{R} \circ \Phi^{-1}\left(y_{j}\right)\right|^{q}\right],  \notag\\
\leq & \E\left[\frac{1}{n^{q}} \cdot n^{q-1} \sum_{j=1}^{n}\left|f \circ \Phi^{-1}\left(y_{j}\right)-f \circ P_{R} \circ \Phi^{-1}\left(y_{j}\right)\right|^{q}\right] \notag\\
= & \frac{1}{n} \sum_{j=1}^{n} \E\left|f(X)-f \circ P_{R}(X)\right|^{q} \label{term:identicallydistinProof}\\
= & \E\left|f(X)-f \circ P_{R}(X)\right|^{q},  \notag
\end{align}
where we use each $\Phi^{-1}(y_j) \sim N(\boldsymbol{0},I_d)$ in \cref{term:identicallydistinProof}. For $E_3$ in \cref{term:ErrorDecompositionInAppendix}, by Jensen's inequality, we have
\begin{equation}
\left|\E\left[f(X)-f \circ P_{R}(X)\right]\right|^{q} \leq \E\left|f(X)-f \circ P_{R}(X)\right|^{q}, \notag
\end{equation}
Therefore, using \cref{lemma:ProjectionError},  we obtain $E_1$ and $E_3$ both can be bounded by 
\begin{align}
& C_0(R-1)^{d-2} \exp \left\{-\frac{1}{4}(R-1)^{2}\right\}. \notag
\end{align}

For the remaining $E_2$,  we obtain the following result.
\begin{lemma}
For any $ f \in G_{e}\left(M, B, K\right)(0<k<2)$ and the RQMC point set $\left\{y_{1}, \ldots, y_{n}\right\}$, we have
\begin{align}
\E\left|\widehat{I}_{n}(f\circ P_R)-E [ f \circ P_{R}(X)]\right|^{q}
& \leq 2^{2qd} B^q R^{qd} \exp \left\{qM(\sqrt{d} R)^{k}\right\}  \frac{C(\log n)^{q(d-1)}}{n^q}. \notag
\end{align}
\end{lemma}

\begin{proof}
By the lemma 4.5 in \cite{Ouyang2024}, we have
\begin{align}
\left|\widehat{I}_{n}(f\circ P_R)-E\left[f \circ P_{R}(x)\right]\right| 
& \leq V_{\text{HK}}\left(f \circ P_{R} \circ \Phi^{-1}\right)  D_{n}^{\ast}\left\{y_{1} \ldots y_{n}\right\} \notag\\
& \leq 2^{2 d} B R^{d} \exp \left\{M(\sqrt{d} R)^{k}\right\}  \frac{C(\log n)^{d-1}}{n}. \notag
\end{align}
Thus we obtain
\begin{align}
\E\left|\widehat{I}_{n}(f\circ P_R)-E [ f \circ P_{R}(X)]\right|^{q}
& \leq 2^{2qd} B^q R^{qd} \exp \left\{qM(\sqrt{d} R)^{k}\right\}  \frac{C(\log n)^{q(d-1)}}{n^q}. \notag
\end{align}
\end{proof}

Now we turn to the proof of \cref{thm:Lq-convergence}.
\begin{proof}[Proof of \cref{thm:Lq-convergence}]
Using \cref{term:ErrorDecompositionInAppendix} and combining the results of the previous three errors $E_1,E_2$ and $E_3$, for any $ f \in G_{e}\left(M, B, k\right)$ with $0<k<2$ and any $R>\sqrt{2}+1$, we can bound $\E\left|\widehat{I}_{n}(f)-E[f(X)]\right|^{q}$ by
\begin{align}
C_{1}(R-1)^{d-2} \exp \left\{-\frac{1}{4}(R-1)^{2}\right\}+C_{2} R^{q d} \exp \left\{qM(\sqrt{d} R)^{k}\right\} \frac{(\log n)^{q(d-1)}}{n^{q}}, \notag
\end{align}
where the constant $C_1 = 3^{q-1}2C_0$
and $C_2 = 3^{q-1}2^{2qd} B^q C.$

Finally, we end the proof by setting different $R$ for different sample number $n>1$.
When $0<k\leq1$, setting $R=(4 q \log n)^{\frac{1}{2}}+1,$ we get 
\begin{align}
& \E\left|\widehat{I}_{n}(f)-\E[f(X)]\right|^{q} \notag\\
\leq & C_{1}\frac{\left(4q \log n\right)^{\frac{d-2}{2}}}{n^{q}} + {C}_{2}C_3(4 q \log n)^{\frac{q d}{2}} \exp \left\{q M(4 q d \log n)^{\frac{k}{2}}\right\} \frac{(\log n)^{q(d-1)}}{n^{q}} \notag\\
= & \mathcal{O}\left(n^{-q}(\log n)^{q\left(\frac{3 d}{2}-1\right)} \exp \left\{q M(4 q d \log n)^{\frac{k}{2}}\right\}\right),\notag
\end{align}
where ${C}_3$ is from the boundedness of $\frac{f(x+1)}{f(x)}$ for $f(x) = x^{q d} \exp \left\{qM(\sqrt{d} x)^{k}\right\}$.\\
When $1<k<2$, setting $R=(4 q \log n)^{\frac{1}{2}}$, we get
\begin{align}
& \E\left|\widehat{I}_{n}(f)-\E[f(X)]\right|^{q} \notag\\
\leq & C_1R^{d-2}\exp\left\{-\frac{R^2}{4}+\frac{R}{2}-\frac{1}{4}\right\} + C_{2} R^{q d} \exp \left\{q M(\sqrt{d}R)^{k}\right\} \frac{(\log n)^{q(d-1)}}{n^{q}} \notag\\
\leq & {C}_1C_4 R^{d-2}\exp\left\{-\frac{R^2}{4}\right\}\exp\left\{q M(\sqrt{d}R)^{k}\right\} + C_{2} R^{q d} \exp \left\{q M(\sqrt{d}R)^{k}\right\}  \frac{(\log n)^{q(d-1)}}{n^{q}} \notag\\
= & \mathcal{O}\left(n^{-q}(\log n)^{q\left(\frac{3 d}{2}-1\right)} \exp \left\{q M(4 q d \log n)^{\frac{k}{2}}\right\}\right) .\notag
\end{align}
where ${C}_4$ is from the boundedness of $f(x) = \exp\left\{\frac{x}{2}-\frac{1}{4}-qM(\sqrt{d} x)^{k}\right\}$.

Consequently, for $0<k<2$, we obtain that
\begin{equation}
\sup_{f \in G_e(M,B,k)} \E\left[\left\lvert \widehat{I}_{n}(f)-\E[f(X)]\right\lvert^{q}\right]=\mathcal{O}\left(n^{-q}(\log n)^{q\left(\frac{3 d}{2}-1\right)} \exp \left\{q M(4 q d \log n)^{\frac{k}{2}}\right\}\right),\notag 
\end{equation}
where the constant in big-$\mathcal{O}$ bound depends on $M,B,k,d,q$.
\end{proof}


\bibliographystyle{plain}
\bibliography{arXiv_references}

\end{document}